\documentclass[10pt]{amsart}

\usepackage{verbatim}

\usepackage{latexsym}
\usepackage{enumerate}
\usepackage{amsmath,amssymb,amsthm,amsfonts,latexsym}
\usepackage{amsmath,amscd}
\usepackage{pstricks, pst-node, pst-text, pst-3d}
\usepackage[bookmarks]{hyperref}
\hypersetup{backref, colorlinks=true}

\let\=\partial

\newtheorem {prop}{Proposition} [section] 
\newtheorem {thm}[prop]{Theorem}
\newtheorem {cor}[prop]{Corollary}
\newtheorem{lem}[prop]{Lemma}

\theoremstyle{definition}
 \newtheorem {rk}[prop]{Remark}
\newtheorem {df}[prop]{Definition}

\newtheorem {ex}[prop]{Example}
\newtheorem {notation}[prop]{Notation}

\newcommand{\Q} {\mathbb{Q}}

\newcommand{\R} {\mathbb{R}}

\newcommand{\N} {\mathbb{N}}
\newcommand{\C} {\mathcal{C}}

\newcommand{\eps}{\varepsilon}

\newcommand{\A}{\mathcal{A}}

\newcommand{\U}{\mathcal{U}}
\newcommand{\F}{\mathcal{F}}

\newcommand{\hk}{(H_k;\lambda_k)_{1 \leq k \leq b}}
\newcommand{\pa}{\partial}

\makeatletter

\author{Leonid Shartser}

\address
{Department of Mathematics, University of Toronto, 40 St. George
st, Toronto, ON, Canada M5S 2E4 }

\email{shartl@math.toronto.edu}


\thanks{}


\begin{document}
\title[On $L^p$ inequality of a semialgebraic set ]{On $L^p$ inequality for differential forms and $L^p$ cohomology of a semialgebraic set for $p>>1$}
\maketitle
\begin{abstract}
We study  Poincar\'e type $L^p$ inequality on a compact semialgebraic subset of 
$\R^n$ for $p>>1$. First we derive a local inequality by using a Lipschitz
deformation retraction with estimates on its derivatives. Then, we extend
the local inequality to a global inequality by employing double complex technique. 
As a consequence we obtain an isomorphism between $L^p$ cohomology and
singular cohomology of a normal compact semialgebraic pseudomanifold.
\end{abstract}

\setcounter{tocdepth}{1}
\tableofcontents

\section{ Introduction }
Let $X\subset\R^n$ be a compact semialgebraic set and $\omega$ a smooth $k$-form on $X_{reg}$, the regular part of $X$.  We say that  $\omega$  is $L^p$ bounded when
$$ \| \omega \|_{L^p} := \left(\int_{X_{reg}} |\omega(x)|^p dVol(x) \right)^{1/p} <\infty\ .$$
Assume $\omega$ is a closed $L^p$ bounded smooth $k$-form on $X_{reg}$.
We prove that if all of the integrals of $\omega$  on the cycles in $X$ vanish
(see Section \ref{singular_case_2} for the precise definition of an integral of an $L^p$ bounded form on a cycle in $X$)  then there exists a smooth $(k-1)$-form $\xi$ such that $\omega= d\xi$ and, moreover, 
\begin{equation} \label{eq11}
\| \xi \|_{L^p(X)} \leq C \| \omega \|_{L^p(X)}
\end{equation}
holds for $p>>1$, where $C$ depends only on the set $X$ and $p$. 
Of course for a contractible semialgebraic set $X$, 
there are no cycles in $X$. Consequently there is a $\xi$ such that $\omega=d\xi$ and inequality (\ref{eq11}) holds on $X$ for $p>>1$.

In \cite{S} we proved a generalization of inequality  (\ref{eq11}), but on compact manifolds. 
Namely, we constructed for every smooth exact $k$-form $\omega$ on a compact Riemannian manifold 
$M$, $\dim M =n$, a smooth $(k-1)$-form $\xi$ on $M$ such that $ \omega = d\xi$ and inequality
\begin{equation} \label{eq12}
\| \xi \|_{L^p(M)} \leq C \| d\omega \|_{L^q(M)}
\end{equation}
holds for $p$ and $q$ in the standard range (i.e. $p<q$ or $p\geq q$ and $\frac{1}{q}-\frac{1}{p}<\frac{1}{n}$) with a positive
constant $C$ depending only on $p$, $q$, $k$ and the manifold $M$.
We proved (\ref{eq12}) first on a convex set following the arguments of Lemma 3.11 of \cite{BoMi} and then derived the global version by means of a method (suggested to us by P. Milman) based on the Weil's double complex. The local version of (\ref{eq12}) appeared in \cite{IwLu}.
In this article we will make use of the 'globalization' method of the proof of (\ref{eq12}) in \cite{S}.

The main difficulty in extending the proof of (\ref{eq12}) to a neighborhood of a point in a set with singularities is that we can no longer connect any two points by a straight line that lies
entirely in the set.

To overcome this difficulty we make use of a Lipschitz deformation retraction $r$ to a single point with estimates on its derivatives, namely:

\smallskip
\noindent
{\bf Theorem \ref{lip_retraction_main}}{\it\ (Lipschitz deformation retraction theorem)
Let $\Sigma_0$ be a stratification of $\R^n$, $X=\cup X_j$, $X_j\in\Sigma_0$, $0\in\overline X_j\cap X$, $j=1,\dots,m$.
There exist a stratified neighborhood $(U,\Sigma_U)$ of $0$ in $\R^n$ with $\Sigma_U$ a cell subdivision such that $\Sigma_U\prec\Sigma\cap U$ and a Lipschitz semialgebraic deformation retraction $r:U\times I\to U$, 
such that 
}
\begin{enumerate}
\item $r_0(x)=0$, $r_1(x)=x$ ,
\item $r|_{S\times(0,1]}$ is smooth ,
\item $|\det Dr_t|\gtrsim t^\mu$, for some $\mu\geq 0$,
\item $\|Dr_t \|\lesssim t^\lambda$ for some $\lambda>0$,
\end{enumerate}
{\it
where $r_t(x):=r(x,t)$, $Dr_t$ is the tangent map of $r_t$ and
$\|Dr_t \|$ denotes the operator max-norm of the tangent map.
}
\smallskip

By means of the latter we define a homotopy operator $R$ such that for a closed form $\omega$ we have 
$\omega=d R\omega$ and our estimates on the derivatives of the deformation retraction $r$ allow us 
to prove that for $p>>1$ our homotopy operator $R$ is an $L^p$ bounded operator. 
Consequently we conclude that $\xi:=R\omega$ is the solution to inequality (\ref{eq11})
on a neighborhood of a point in $X$.

Fortunately, 'globalization' of the local $L^p$ inequality (\ref{eq11}) 
to a semialgebraic set can be carried out essentially just like for a smooth manifold in \cite{S}. The basic two facts that are needed
for proving this global version is 
the validity of the local version of inequality (\ref{eq11}) and the existence
of a partition of unity (with locally bounded differentials), which semialgebraic sets admit. 


One of the most important applications of the local version of inequality (\ref{eq11}) is in 
the theory of $L^p$ cohomology on semialgebraic sets. 
To define $L^p$ cohomology we consider a differential complex
consisting of the $L^p$ bounded forms with $L^p$ bounded weak exterior derivatives on the regular part of the set in question. $L^p$ cohomology is defined as the factor space of closed $L^p$ bounded forms by the exact $L^p$ bounded forms. 
Of course for compact semialgebraic sets, $L^p$ cohomology is an invariant of 
the induced metric. 
But the question of finiteness of the latter in general (for any $p\leq\infty$) was open.
The $L^p$ cohomology theory is addressed  
in several special cases by various authors (see e.g. \cite{Ch},\cite{Y},\cite{HP},\cite{GKS},\cite{GKS2},\cite{GKS3},\cite{Gr}).

In this article we show that as a consequence of inequality (\ref{eq11}) 
 $L^p$ Poincar\'e lemma is valid for $p>>1$ and hence $L^p$ cohomology coincides with
the singular cohomology of a compact (normal) semialgebraic set.

\subsection{Organization of the article}
In Section \ref{singular_case} we prove a local inequality of the form of inequality (\ref{eq11}) for smooth $L^p$ bounded forms in a neighborhood of a point in a semialgebraic set.

In Section \ref{Lp_cohomology1} we give an application of the local $L^p$ inequality to 
$L^p$ cohomology of a compact normal semialgebraic set $X$. We show that for $p>>1$ the $L^p$
cohomology of such sets coincides with the singular cohomology
by means of a sheaf theoretic argument.

In Section \ref{singular_case_2} we extend the local inequality to a global inequality on 
compact semialgebraic sets. We show that under certain conditions, which we express combinatorially, closed $L^p$ bounded forms satisfy (\ref{eq11}) for $p>>1$.

Section \ref{sec_Lip_retract} is introduction to the construction of the Lipschitz deformation 
retraction with estimates on derivatives. 

In Section \ref{sec_reg_fam} we introduce some technical material needed for our construction of 
Lipschitz deformation retraction.

Finally, in Section \ref{sec_Lip_ret} we prove the Lipschitz deformation retraction Theorem \ref{lip_retraction_main}.

\vspace*{4mm}
\noindent \textbf{Acknowledgment. }
I would like to thank P. Milman for posing the questions and for contributing
many fundamental ideas that were used in this work.
I would also like to thank G.Valette for teaching me his invaluable Lipschitz geometry techniques.
\vspace*{2mm}

\section{Notations and basic definitions}\label{sec_defs}

Let $X\subset\R^n$  be a semialgebraic set. Denote by $X_{reg}$ the subset of $X$ consisting 
of points where $X$ is a smooth manifold and set $X_{sing}:=X-X_{reg}$. Denote by $\overline {X}$ the closure of $X$ and by $\text{bd }X$ the topological boundary of $X$.
\begin{itemize}

\item $(\Omega^\bullet(X_{reg}),d)$ denotes the complex of smooth $k$-forms on $X_{reg}$ and 
with exterior derivative $d:\Omega^k(X_{reg})\to\Omega^{k+1}(X_{reg})$.

\item For a form $\omega\in\Omega^k(X_{reg})$ define the $L^p$ norm by
$$ \| \omega \|_{L^p} := \left(\int_{X_{reg}} |\omega(x)|^p dVol(x) \right)^{1/p} <\infty,$$  
where $|\omega(x)|$ is the pointwise norm of $\omega$ at the point $x\in X_{reg}$
defined by 
$$ \sup_{v\in\wedge^{k}(X_{reg})}\frac{|\omega(x;v)}{|v|} .$$

\end{itemize}

Suppose that $X_{reg}$ is of dimension $n$ and $\omega\in \Omega^k_{L^p}(X_{reg})$. 
A form $\gamma\in\Omega^{k+1}_{L^p}(X_{reg})$ is said to be
{\bf the weak exterior derivative} of $\omega$ if for every point $p\in X_{reg}$ there
exists a neighborhood $U$ such that for every smooth $(n-k-1)$-form $\phi$ supported
in $U$ we have 
$$ \int_{U} \omega\wedge d\phi = (-1)^{k+1}\int_{U} \gamma \wedge \phi.$$
The weak exterior derivative of $\omega$ is denoted by $\overline d\omega$.


%
%
%

\section{Local $L^p$ inequality on a semialgebraic set}\label{singular_case}
Let $X\subset \R^n$ be a compact semialgebraic set with $a\in X$.
Denote by $(\Omega^\bullet_{L^p}(X_{reg}),\overline d)$ the complex of $L^p$ bounded forms with $L^p$ bounded weak exterior derivatives, i.e., forms $\omega$ with 
$$ \|\omega \|_{L^{p,1}}:=\|\omega \|_{L^{p}} + \| \overline d\omega\|_{L^p} <\infty. $$
We say that $X$ admits a {\bf local $L^p$ estimate} near $a$ if 
there is a neighborhood $U$ of $a$ in $X$ such that 
for every closed smooth $L^p$ bounded $k$-form $\omega$, $k\geq 1$, defined in $U$
 there is a smooth form $\xi$, defined in $U$,  such that 
\begin{equation}\label{loc_ineq}
\left\{
\begin{array}{lc}
 \omega=d\xi\ \ \text{in $U$}, \\ \|\xi\|_{L^p(U)}\leq C \| \omega \|_{L^p(U)}
\end{array}
\right.
\end{equation}
where $C>0$ is independent of $\omega$.

We prove in this section that $X$ admits local $L^p$ estimate for $p>>1$.
The main technical tool is  our Lipschitz deformation retraction Theorem \ref{lip_retraction_main}.

\subsection{Homotopy Opertator.}
Let $(U,\Sigma)$ be a stratified neighborhood and $r : U\times I \to U$, $I:=[0,1]$, be the Lipschitz semialgebraic deformation retraction obtained by applying Theorem \ref{lip_retraction_main} to the set $X$ and any stratification 
of $\R^n$ that is compatible with $X$.
Let $\eps>0$. We associate a homotopy operator $R_{\eps}$ with the deformation retraction $r$  as follows:\\
Let $\alpha$ be an $L^p$ bounded smooth $k$-form on $X_{reg}$. The pull back $r^*\alpha$ 
is a form on $U\times I$ and can be represented as $\alpha_0+dt\wedge\alpha_1$
where $t$ is the coordinate in $I$. 
Define an operator 
$$P:\Omega^{k}_{L^p}(U)\to \Omega^{k-1}_{L^p}(U\times I),\ \ P\alpha:=\alpha_1.$$
Set 
$$ R_\eps\alpha := \int_\eps^1 \alpha_1(x,t)dt .$$

Observe that $R_{\eps}\alpha$ is defined almost everywhere on every stratum of $\Sigma$ that is contained in $U$.
Next we show that $R_\eps$ is an $L^p$ bounded operator (for $p$ large enough) and therefore $R_\eps \alpha$ defines an element in $L^p$.
We will need the following lemma.
\begin{lem}
Suppose that $S\subset\R^n$ is a locally closed oriented submanifold of dimension $k$ and 
$\phi:D\to S$ is a bi-Lipschitz diffeomorphism from an open and bounded domain $D\subset\R^k$.
Then, 
$$ \int_S f(x) dVol(x) \sim \int_D f(\phi(x))dx_1\dots dx_k \text { for any } f:S\to \R\  ,$$
where $dVol(x)$ is the volume form on $S$ and $x_1,\dots,x_k$ are coordinates in $D$.
\end{lem}
\begin{proof}
In coordinates $x=(x_1,\dots,x_k)$ on $D$ and $y=(y_1,\dots,y_n)$ on $\R^n\supset S$ the description of $\phi$ is 
$ y_1 = \phi_1(x),\dots,y_n=\phi_n(x)$.
Also 
$$ \frac{\pa}{\pa x_i}:=\sum_{j} \frac{\pa y_j}{\pa x_i}\frac{\pa}{\pa y_j}\in T_y S\subset\R^n, \ 1\leq i \leq k, $$
is a basis of tangent vectors to $S$ at the point $y=\phi(x)$. 
Thus the volume form $dVol(x)$ in the induced from $\R^n$ Riemannian metric on $S\subset\R^n$ can be written as 
$$ dVol(x) = \sqrt{\det \left<\frac{\pa }{\pa x_i},\frac{\pa }{\pa x_j}\right>_{i,j}}dx_1\dots dx_k\ ,$$
where $\left<\cdot,\cdot \right>$ is the standard scalar product in $\R^n$.
Since $\phi$ is bi-Lipschitz, the 'volume density' function $\sqrt{\det \left<\frac{\pa }{\pa x_i},\frac{\pa }{\pa x_j}\right>_{i,j}}$ is bounded and is non vanishing. Therefore this function is 'equivalent' to a constant. 
\end{proof}
\begin{thm}(Local $L^p$ inequality Theorem)\label{thm_loc_lp_ineq}
Suppose that $\omega$ is a smooth $L^p$ bounded $k$-form, $k\in\N$, defined on $X_{reg}$ near $0\in X$. 
Then there is a neighborhood $U$ of $0\in X$ 
such that 
\begin{enumerate}[(i)]
\item $\| R_\eps \omega\|_{L^p(U)}\leq C \| \omega \|_{L^p(U)}$ ,
\item $R_\eps\omega \to R_0 \omega $ in $L^p$,
\item $\| r^*_\eps \omega\|_{L^p(U)}\to 0 $ as $\eps\to 0$.
\end{enumerate}
where $p>>1$ and $C>0$ depend only on the set $X$.
\end{thm}
\begin{proof}
Let $(U,\Sigma)$ be a stratified neighborhood of $0\in X$ and $r:U\times I\to U$ be a Lipschitz deformation retraction given by Theorem \ref{lip_retraction_main}. Clearly it is enough to prove Theorem \ref{thm_loc_lp_ineq} for every $S\in \Sigma$ of dimension $\dim X$ which 
is a stratum contained in $U$. So let $S$ be such a stratum.
Let $\omega_1:=P\omega$. Then
\begin{eqnarray*}
\|R_\eps\omega \|_{L^p(S)} &=& \|\int_\eps^1 \omega_1(x,t) dt \|_{L^p(S)}\ .
\end{eqnarray*}
Note that $\omega_1(x,t;\cdot) =r^*\omega(x,t; {\frac{\pa }{\pa t}},\cdot )$ or equivalently, for every
$v\in\wedge^{k-1}(\R^n)$ 

$$ \omega_1((x,t);v)= \omega(r(x,t); r_*\frac{\pa}{\pa t}\wedge r_*v)\ , $$
holds, where $r_*$ denotes the push forward map of the deformation retraction $r$.

According to Theorem \ref{lip_retraction_main} there is $\lambda>0$ such that an upper bound $\|Dr_t\|\lesssim t^{\lambda}$ holds. 
It follows that 
\begin{eqnarray*}
\left|\omega_1(x,t) \right| &=& \sup_{|v|=1}\left|\omega(r(x,t);r_*{\frac{\pa }{\pa t}}\wedge r_*v)\right|\\
&\leq&
\left|\omega(r(x,t)\right| \sup_{|v|=1}\left|r_*{\frac{\pa }{\pa t}}\wedge r_*v)\right|\\
&\leq& C\left|\omega(r(x,t)\right| \| Dr_t \|^{k-1}\\ 
&\leq& C\left|\omega(r(x,t)\right| t^{(k-1)\lambda}\ .
\end{eqnarray*} 
Consequently 
\begin{eqnarray*}
\|\int_\eps^1 \omega_1(x,t) dt \|_{L^p(S,dx)} &\lesssim&  
\|\int_\eps^1 t^{(k-1)\lambda}|\omega(r(x,t))| dt \|_{L^p(S,dx)} \\
&\leq&
\int_\eps^1 \| |\omega(r(x,t))|  \|_{L^p(S,dx)} t^{(k-1)\lambda}dt\\
  &=& \int_\eps^1 \| |\omega(z)| \left| \frac{\pa x}{\pa z}\right|^{\frac{1}{p}}\|_{L^p(S,dz)} t^{(k-1)\lambda}dt\\
\end{eqnarray*}
and since $\left|\frac {\pa z}{\pa x}\right| \gtrsim t^{\mu}$ for $z:=r(x,t)$, the upper bound on the latter is 
\begin{eqnarray*}
  &\leq& \int_\eps^1 t^{-\frac{\mu}{p}+{(k-1)\lambda}}\| |\omega(z)| \|_{L^p(S,dz)} dt\\
  &=& \frac{p}{p\left(1+(k-1)\lambda\right)-\mu}\left(1-\eps^{-\frac{\mu}{p}+{(k-1)\lambda}+1}\right)
  \|\omega\|_{L^p(S)}\ .                           
\end{eqnarray*}

Therefore, for any $p>\frac{\mu}{1+(k-1)\lambda}$ and any $\eps\geq 0$ the homotopy operator $R_\eps$ is a bounded operator between the $L^p$ spaces of differential forms.

%
%
%
%
For part $(ii)$, we have to show that if $\eps,\eps'$ are small then 
$\| R_\eps\omega - R_{\eps'}\omega \|_{L^p(S)}$ is small.
The same type of computation as in the previous paragraph (replacing integration from $\eps$ to $1$
by integration from $\eps'$ to $\eps$) implies part $(ii)$.

To prove part $(iii)$ we will make use of
\begin{eqnarray*}
\| r^*_\eps\omega \|^p_{L^p(S)} &=& \int_{S} |r^*_\eps\omega|^p dx \\
 &\leq& \int_{S} |\omega(r_\eps(x))|^p \|Dr_\eps\|^{pk} dx \\
 &\leq& \int_{S} |\omega(r_\eps(x))|^p \eps^{pk\lambda} dx \\
\end{eqnarray*}
and with $z=r_\eps(x)$ the upper bound for the latter is
\begin{eqnarray*}
 &\leq& \int_{S} |\omega(z)|^p \left|\frac{\pa x}{\pa z}\right|\eps^{pk\lambda} dz \\
 &\lesssim& \int_{S} |\omega(z)|^p \eps^{-\mu}\eps^{pk\lambda} dz \\
\end{eqnarray*}
Therefore $p>\frac{\mu}{k\lambda}$ implies $\| r^*_\eps \omega \|_{L^p} \to 0 $ as $\eps \to 0$,
as required.
%
%
%
%
%
%
%

\end{proof}

Next, we show that $R_\eps$ satisfies the classical homotopy identity.

\begin{prop}\label{homotopy_prop}
Suppose that $V$ is a stratified neighborhood of $0\in X$ provided by Theorem \ref{lip_retraction_main} and consequently the operator $R_\eps$ is defined on $V$.
The homotopy operator $R_\eps$ satisfies the following homotopy identity:
$$ \overline{d}R_\eps\alpha + R_\eps d\alpha=\alpha -r_\eps^*\alpha $$
for any smooth $L^p$ bounded form  $\alpha$ defined on $U:=X_{reg}\cap V$. 
\end{prop}
\begin{proof}
In order to check this identity let $\phi$ be a smooth $(n-k)$-form with a compact support
in $U$. We have to show that
\begin{equation}\label{weak_homotopy_formula}
\int_{U} (\alpha-r^*_\eps\alpha-R_\eps d\alpha)\wedge \phi = (-1)^{k}\int_U R_\eps\alpha \wedge d\phi.
\end{equation}
%
%
%
%
%
%
Note that since $r$ is Lipschitz and $\alpha$ is smooth on $X_{reg}$ the pullback $r^*\alpha$ of 
the form $\alpha$ is an
$L^\infty$ form on $X$ in the sense of \cite{SV}  on 
$$ \{x\in S: d(x,X_{sing}\cap S)\geq \eps\}\times (0,1]$$ 
for every stratum $S\in\Sigma$, $S\subset X_{reg}$ and any $\eps>0$ .
As a result, 
the forms $\alpha_1\wedge\phi$ and $\alpha_1\wedge d\phi$, where $\alpha_1:=P\alpha$ are also $L^\infty$ forms in the sense of \cite{SV}.
To be precise we mean that there is a stratification $\C$ of $U\times [\eps,1]$ such that
the forms $\alpha_1\wedge\phi$ and $\alpha_1\wedge d\phi$ are stratified and bounded with 
stratified and bounded exterior derivatives.
In the computation below exterior derivative of a form is calculated on each stratum separately.

Below we denote by $d_x$ the exterior derivative with respect to $x$ (ignoring the variable $t$). We begin by analyzing the right hand side of (\ref{weak_homotopy_formula}):
\begin{eqnarray}\label{verif_1}
\int_U R_\eps\alpha \wedge d\phi &=& \int_U \int_\eps^1 \alpha_1(x,t)dt \wedge d\phi\nonumber \\
&=& \int_{U\times[\eps,1]} \alpha_1(x,t)dt \wedge d\phi \\
&=& (-1)^k\int_{U\times[\eps,1]} d(\alpha_1(x,t)dt \wedge \phi) - d_x\alpha_1dt\wedge\phi\nonumber \\
&=& (-1)^k\int_{U\times[\eps,1]} d(\alpha_1(x,t)dt \wedge \phi) - (-1)^k\int_{U\times[\eps,1]}d_x\alpha_1dt\wedge\phi \nonumber \\ 
&=& (-1)^{k+1}\int_{U\times[\eps,1]}d_x\alpha_1dt\wedge\phi.\nonumber
\end{eqnarray}
The latter equality makes use of the Stokes' formula for $L^\infty$ forms. Indeed,
\begin{eqnarray*}
\int_{U\times[\eps,1]} d(\alpha_1(x,t)dt \wedge \phi) &=& \int_{\pa( U\times[\eps,1])} \alpha_1(x,t)dt \wedge \phi\\
&=& 
\int_{\pa U\times[\eps,1]} \alpha_1(x,t)dt \wedge \phi + (-1)^n\int_{U\times\pa [\eps,1]} \alpha_1(x,t)dt \wedge \phi\\
\end{eqnarray*}
The first summand in the latter equation equals to zero since $\phi$ is compactly supported in $U$
and the second summand vanishes since on $U\times\pa[\eps,1]$  the variable $t$ is locally constant and therefore $dt=0$.

Next we simplify the left hand side of the equation (\ref{weak_homotopy_formula}):
$$\int_{U} (\alpha-r^*_\eps\alpha-R_\eps d\alpha)\wedge \phi =
\int_{U} (\alpha-r^*_\eps\alpha)\wedge\phi-\int_{U\times[\eps,1]} \left(\frac{\pa\alpha_0}{\pa t}-d_x\alpha_1\right)dt\wedge \phi\ , $$

where $\alpha_0:=r^*\alpha - P\alpha$.
Formula (\ref{weak_homotopy_formula}) follows from
$$\int_{U} (\alpha-r^*_\eps\alpha)\wedge\phi=\int_{U\times[\eps,1]} \frac{\pa\alpha_0}{\pa t}dt\wedge \phi, $$
which we prove below.

By refining the stratification $\C$ we may assume that 
each stratum $S\in\C$ is a cell in $\R^n\times[\eps,1]$. In particular, 
it means that the projection of each cell $S\in\C$ to the first $n$ coordinates is a cell
in $U$. Hence, we may assume that $\C=\{S_{i,j}\}$, where $i,j\in \N$, and
$$ S_{i,j}:=\{(x,t)\in U\times[\eps,1]: \eta_{i,j}(x)\leq t\leq \eta_{i,j+1}(x), x\in S'_i\}\ ,$$
where $S'_i$ is the projection to the first $n$ coordinates of the set $S_{i,j}$ (for any $j$) and
$\eta_{i,j}$ are smooth semialgebraic functions defined over $S'_i$.

Next, since $r^*\alpha$ is stratified it follows that $r^*\alpha|_{t=\eta_{i,j}(x)}$ is well defined for $x\in S'_i$ and, in particular, $\alpha_0(x,\eta_{i,j}(x))$ is well defined.
Therefore, 
\begin{eqnarray*}
\int_{U\times[\eps,1]} \frac{\pa\alpha_0}{\pa t}dt\wedge \phi &=&
\int_{U}\left(\int_{\eps}^{1} \frac{\pa\alpha_0}{\pa t}dt\right)\wedge \phi\\
&=& \sum_{i,j} \int_{S'_i}\left( \int_{\eta_{i,j}(x)}^{\eta_{i,j+1}(x)} \frac{\pa\alpha_0}{\pa t}dt\right)\wedge\phi\\
&=& \sum_{i} \int_{S'_i}\left(\alpha_0(x,1) - \alpha_0(x,\eps) \right)\wedge\phi\\
&=& \int_{U} (\alpha-r_\eps^*\alpha)\wedge\phi\ ,
\end{eqnarray*}
as required.

\end{proof}


%
%
%
%
%
\subsection{Finding a smooth solution to problem (\ref{loc_ineq})}
As a corollary of the results of the previous section the following holds 
\begin{cor}\label{h_cor}
In the setting of Proposition \ref{homotopy_prop} assume $\omega$ is a smooth $L^p$ bounded closed form defined on $U$. Then there exists a smooth $L^p$ bounded form $\xi$ solving problem (\ref{loc_ineq}).
\end{cor}

To prove this corollary we will need a theorem from \cite{Y}.
\begin{thm}(Theorem 2.7.1 \cite{Y})\label{Y1}
Let $M$ be Riemannian manifold. Suppose that $\omega$ in an $L^p$ bounded form on $X$ with 
$\overline d \omega$ a smooth $L^p$ bounded form. Then for any $\eps>0$ there exits a form $\psi_\eps$ such that $\|\psi_\eps \|_{L^{p}}+\|\overline d\psi_\eps \|_{L^{p}}<\eps$ and $\omega+\overline d\psi_\eps$ is smooth.
\end{thm}
\begin{proof}[Proof of Corollary \ref{h_cor}]
Denote $R_\eps \omega$ by $ \xi'_\eps$.
According to $(ii)$ of Theorem \ref{thm_loc_lp_ineq} $\xi':=\lim_{\eps\to 0} \xi'_\eps $
is $L^p$ bounded. Proposition \ref{homotopy_prop} implies
$$ \overline d\xi'_\eps = \omega - r^*_{\eps}\omega. $$
Hence $(iii)$ of Theorem \ref{thm_loc_lp_ineq} and passing to limit 
as $\eps \to 0$ imply $ \overline d\xi' = \omega. $
Moreover, $(i)$ of Theorem \ref{thm_loc_lp_ineq} implies 
$\| \xi'\|_{L^{p}} \leq C \| \omega\|_{L^{p}}$.
According to Theorem \ref{Y1} there is a form $\psi$ such that 
$\|\psi \|_{L^{p}}+\|\overline d\psi \|_{L^{p}} < \|\omega\|_{L^p}$ and $\xi:=\xi'+\overline d\psi$ is smooth.
Therefore $d\xi = \overline d\xi'$ and
$$ \| \xi \|_{L^p} = \| \xi' + \overline d\psi \|_{L^p} \leq (C+1)\| \omega \|_{L^p}, $$
as required.
\end{proof}

\section{$L^p$-cohomology}\label{Lp_cohomology1}

In this section we consider an $L^p$ cohomology theory of a normal compact semialgebraic set $X$.

The $L^p$ cohomology is the cohomology of the complex $(\Omega^\bullet_{L^p}(X),\overline d)$  commonly defined by 
$$ H^k_{L^p}(X):= \frac{\text{Ker }(\overline d: \Omega^k_{L^p}(X_{reg})\to \Omega^{k+1}_{L^p}(X_{reg}))}
{\text{Im }(\overline d: \Omega^{k-1}_{L^p}(X_{reg})\to \Omega^{k}_{L^p}(X_{reg}))}\ .$$

Let 
$$ \Lambda^k_{L^p}(X):= \Omega^k(X_{reg})\cap \Omega^k_{L^p}(X_{reg}), $$
and denote the $k^{th}$ cohomology group of $(\Lambda^\bullet_{L^p},d)$ by $H^k(\Lambda_{L^p}^\bullet(X))$.

Then Theorem \ref{Y1} implies that 
the cohomology of $\Lambda_{L^p}^\bullet(X)$ is isomorphic to the $L^p$ cohomology:

\begin{prop}
$H^k(\Lambda_{L^p}^\bullet(X))=  H^k_{L^p}(X)$.
\end{prop}
\begin{proof}
We define a homomorphism 
$$i: H^k_{L^p}(X)\to H^k(\Lambda_{L^p}^\bullet(X)) $$
as follows. 
Assume $\omega\in \Omega^k_{L^p}(X_{reg})$ is a closed form.
Denote by $[\omega]$  the $L^p$ cohomology class of $\omega$.
According to Theorem \ref{Y1} there is a form $\psi$ with $\|\psi\|_{L^{p,1}}<\infty$  such 
that $\omega + \overline d\psi$ is smooth.
  
Set $i[\omega]$ to be the $\Lambda_{L^p}$ cohomology class of $\omega + \overline d\psi$.
First note that $i$ is well defined. Indeed, since if $\psi'$ is another form 
such that $\|\psi'\|_{L^{p,1}}<\infty$
and  $\omega + \overline d\psi'$ is smooth then 
$$\overline d(\psi-\psi') = \omega + \overline d\psi - (\omega + \overline d\psi')    $$  
is a smooth form. Therefore, applying Theorem \ref{Y1} once more, we obtain 
a form $\xi$, $\| \xi \|_{L^{p,1}}<\infty$ such that $\psi-\psi'+\overline d\xi$ is smooth.
Finally,  
$$(\omega + \overline d\psi) - (\omega + \overline d\psi')  = d(\psi-\psi'+\overline d\xi)\ ,$$
as we claimed.
The proof of surjectivity of $i$ is straightforward. The homomorphism $i$ is also injective.
Indeed, if $i[\omega]=0$ then 
there is a form $\psi$, $\|\psi\|_{L^{p,1}}<\infty$ such that 
$\omega + \overline d \psi = d\gamma$
for a smooth form $\gamma$ with $\|\gamma\|_{L^{p,1}}<\infty$ and injectivity follows.

\end{proof}

\begin{df}
We will refer to a $k$ dimensional subset $X\subset \R^n$ as {\bf normal} if for any $x\in X$, 
there exists $\eps>0$ such that $S^{n-1}(x,\eps)\cap X_{reg}$ is connected, where $S^{n-1}(x,\eps)$ is an $(n-1)$-sphere in $\R^n$ centered at $x$ with radius $\eps$.
\end{df}

Since we work with compact sets, $L^p$ boundedness is a local property and hence germs 
of $L^p$ bounded $k$-forms define a sheaf on $X$. Namely, for every open set $U\subset X$
we associate the set $\Omega^k_{L^p}(U\cap X_{reg})$ or $\Lambda^k_{L^p}(U\cap X_{reg})$. We denote the sheaf of $L^p$ bounded $k$-forms by $\Omega^{k}_{L^p}$ and the sheaf of smooth 
$L^p$ bounded forms by $\Lambda^k_{L^p}$.
The sheaves $\Omega^{k}_{L^p}$ and $\Lambda^{k}_{L^p}$ are fine on compact semialgebraic set $X\subset\R^n$. Indeed, every open cover of $X$ can be extended to an open cover of a neighborhood of $X$ in $\R^n$,
on which existence of a partition of unity is evident.

In Section \ref{singular_case} we proved that, locally, smooth closed $k$-forms on a semialgebraic set $X$ are exact for $p>>1$ and $k>0$. If $X$ is normal then closed $0$-forms
are locally constant functions. It follows from here that the sheaf complex $\Lambda^\bullet_{L^p}$
on a normal set $X$ comprises a fine resolution of the constant sheaf $\R$ on $X$ and
therefore, a standard argument from sheaf theory implies that the singular cohomology of
$X$ naturally coincides with the cohomology of $\Omega^\bullet_{L^p}(X)$.
That is we have proven 
\begin{thm}\label{thm_isom}
Let $X\subset \R^n$ be a normal compact semialgebraic set. 
There exists $p>>1$ such that the cohomology of the $L^p$ complex $\Omega^\bullet_{L^p}(X)$ is isomorphic to the singular cohomology of $X$.
\end{thm}

We remark that for small $p$ the isomorphism of Theorem \ref{thm_isom} 
need not hold.
\medskip
\\
\textbf{Example. } Let $X$ be a semialgebraic set such that $(X_{reg},g)$ is a smooth Riemannian manifold diffeomorphic to $M\times(0,1]$, 
where $M$ is a smooth compact $m$-dimensional manifold with a Riemannian metric $g_M$.
Suppose that $g= dr^2 + r^{2\alpha}g_M$, where $r$ is the coordinate in $(0,1]$ and $\alpha\geq 1$. Topologically, $X$ is a cone over the manifold $M$.
Let $\omega$ be a smooth, closed, non exact and radially constant $k$-form on $X_{reg}$, i.e. $\omega$ does not contain 
any terms of the form $dr\wedge\dots$ and the coefficients of $\omega$ are independent of $r$. 
The volume form on $X_{reg}$ is given by $dV=r^{\alpha m}dr\wedge dV_M$, where $dV_M$ stands for the volume form 
on $M$.
The pointwise norm of $\omega$ is given by 
$$|\omega(x,r)| = r^{-k\alpha}|\omega(x,r)|_{M}, \ \ (x,r)\in M\times(0,1]\ ,$$
where $|\cdot|_M$ is the pointwise norm on $M$.
Note that since $\omega$ is closed and radially constant, it is independent of $r$ and, in 
abuse of notation, we will write $\omega(x)$ instead of $\omega(x,r)$.

Clearly $\|\omega\|_{L^p(X)}<\infty$ if and only if $p<\frac{\alpha m+1}{k\alpha}$:
\begin{eqnarray*}
\int_X |\omega(x)|^p dV &=& \int_0^1 \int_M |\omega|^p r^{\alpha m} dr dV_M \\
&\approx& \int_0^1 r^{-\alpha k p +\alpha m} dr < \infty.
\end{eqnarray*}

In fact, an $L^p$ bounded closed (not exact) radially constant form defines a nontrivial cohomology class in $H^k_{L^p}$. Indeed, otherwise assume $\omega=d\xi$, $\|\xi\|_{L^p}<\infty$. 
Since $\omega$ is radially constant $\xi$ must be radially constant as well.
But then $\omega|_M = d\xi|_M$ which contradicts the assumption of $\omega$
not being exact.
However, for $p>>1$ it follows from Theorem \ref{thm_isom} that $H^{k}_{L^p}(X)=0$ for $k>0$.
In our example, the minimal $p$ for which $H^{k}_{L^p}(X)=0$
equals $\frac{\alpha m+1}{k\alpha}$ .
%
%
%
%
%
%
%
%
%
%
%
%
%
%
%
%
%
%
\section{Global $L^p$ inequality on a semialgebraic set}\label{singular_case_2}
For a closed form considered in Section \ref{singular_case} 
the problem of finding an antiderivative with the bound (\ref{loc_ineq}) can be generalized to a global problem on a compact semialgebraic set $X$. 
In the case that $X$ is a compact smooth manifold, such problem is treated in \cite{S}.
In general, due to the existence of topological obstructions, of course there are no 
antiderivatives for some closed forms. 

To overcome this obstacle we derive a combinatorial condition under which closed forms are exact. 
In the case that $X$ is a compact smooth manifold, a closed form on $X$ is exact if and only if its integrals vanish on every cycle in $X$. 
We generalize this condition to closed $L^p$ bounded forms by 
extending the notion of integration over cycles
in $X$ to the case of closed $L^p$ bounded forms. In particular, our generalized condition for closed $L^p$ bounded forms to be exact is the usual one (mentioned above) whenever $X$ is a (nonsingular) manifold. 
Our definition of an integral of a closed $L^p$ bounded form over a cycle in $X$ 
is placed in the forthcoming subsection and is of a 
combinatorial nature. In \cite{S} we derive a combinatorial formula for an integral of a closed
form over a cycle  in a manifold. 
It is constructed iteratively by means of a process of an application of the $L^p$ inequality for forms on various contractible subsets. This process can be carried out for any class of forms that satisfy the $L^p$ inequality for forms on a 'good' (or even 'weakly good') covering by contractible subsets.
We prove in the forthcoming subsection that closed $L^p$ bounded forms satisfy the $L^p$ inequality for forms on contractible sets, 
which would allow us to extend the notion of an integral over cycles to the $L^p$ bounded forms.

\begin{rk}
In a paper by Gol'dshtein, Kuz'minov and Shvedov \cite{GKS}, the authors defined an integral of forms in $W^k_{p,q}$ over any $k$-dimensional manifold parametrized by a Lipschitz map. However, for our purposes, it suffices to define integrals of closed $L^p$ bounded forms just over cycles. 
\end{rk}
\subsection{Definition of an integral of a closed $L^p$ bounded form.}
To define an integral of an $L^p$ bounded form over a cycle we consider the \v{C}ech-De Rham double complex. We refer the reader to \cite{S} for the related definitions and generalities. 

Assume $X\subset \R^n$ is a compact semialgebraic set and let $\U=\{U_i\}_{i=1,\dots,N}$ be 
a finite open cover of $X$.
We associate a differential \v{C}ech complex with values in the $L^p$ bounded $k$-forms to the cover $\U$.
\begin{df}
The \v{C}ech complex with values in $\Omega_{L^p}^k$ we denote
by  $(K^{k,\bullet}(\U,\Omega_{L^p}^k),\delta)$,
where 
$$ \delta : K^{k,j}(\U,\Omega_{L^p}^k)\to K^{k,j+1}(\U,\Omega_{L^p}^k) $$
is defined by 
$$(\delta \varphi)_{i_0,\dots,i_{j+1}} := \sum_{k} (-1)^k \varphi_{i_0,\dots\hat{i_{k}},\dots,i_{j+1}} .$$
%
The 'combined' double complex {\v C}ech and the complex of $L^p$ bounded forms 
is defined by  $(K^\bullet(\U,\Omega^\bullet_{L^p}),D)$
, where 
 $$K^j(\U):=\bigoplus_{l+k=j} C^l(\U,\Omega_{L^p}^k) $$ and 
$D : K^j(\U)\to K^{j+1}(\U) $
is defined by $D:=d+(-1)^l\delta$ on $C^l(\U,\Omega_{L^p}^k)$.
\end{df}
Denote by $H^j(K^\bullet(\U))$ the cohomology of the complex $K^\bullet$ and
denote by $H^j(C^\bullet(\U,\Omega^r_{L^p}))$ the cohomology of the complex $C^\bullet(\U,\Omega^r_{L^p})$.
In \cite{S} a good cover is defined as a cover consisting of convex sets. In this
article we will work with slightly weaker condition on covers namely:

\begin{df}
If $\U=\{U_i\}$ is a cover of $X$, we say that $\U$ is a {\bf weakly good cover}
if each finite intersection of $U_i$'s is contractible.

The {\bf nerve complex} of a cover $\U$ is a simplicial complex $(C(\U)_\bullet,\pa)$ with 
simplex $[I]$ associated with every non empty intersection $U_I$.
The boundary operator $\pa:C_l(\U)\to C(\U)_{l-1}$ is defined as usually 
$$\pa[I]:=\sum_{j>0} (-1)^j [i_0,\dots,\hat{i_j},\dots,i_l], \ \ I=[i_0,\dots,i_l].$$
\end{df}

It is a well known fact that if $\U$ is a weakly good cover of $X$ then the homology of
the nerve complex $C_\bullet(\U)$ coincides with the singular homology of $X$ (see e.g. \cite{H} Corollary 4G.3).

\begin{rk}
Every triangulable set $X$ has a weakly good cover. Indeed, if $T$ is a triangulation of $X$ 
with $V:=\{1,\dots,|T|\}$ being the set of vertices of $T$, then let $\U:=\{U_i\}_{i\in V}$
be a cover of $X$, where $U_i$ is the star of vertex $i$. We claim that $\{U_i\}_{i\in V}$ is a weakly good cover.
Let $I:=(i_0,\dots,i_l)$ and assume that $U_I:= U_{i_0}\cap\dots\cap U_{i_l} \neq \emptyset$.
Every simplex of dimension $\dim X$ in the closure of $U_I$ contains the vertex $i_0$.
Therefore, it is possible to deformation retract the closure of $U_I$ to $i_0$.
Consequently, every finite intersection $U_I$ is contractible.
\end{rk}

\begin{lem}\label{loc_glob_1}
Assume that $X$ is a compact contractible semialgebraic set.
There exists $p>>1$ such that for every closed $k$-form $\omega$ in $\Lambda^k_{L^p}(X_{reg})$,  $k\geq 1$ there exists a form $\xi\in\Omega^{k-1}_{L^p}(X)$ 
such that 
\begin{equation} \label{loc_1}
\left\{
\begin{array}{lc}
 \omega=d\xi\ \ \text{on $X_{reg}$}, \\ \|\xi\|_{L^p(X)}\leq C \| \omega \|_{L^p(X)}
\end{array}
\right.
\end{equation}
\end{lem}
\begin{proof}
The proof is by induction on $k$. When $k=1$ Corollary \ref{h_cor} implies that there is a cover $\{U_i\}_{i=1}^N$, $ X= \cup_{i=1}^{N} \overline U_i$
such that (\ref{loc_1}) holds on $U_i$ with $\xi$ being $\tilde\xi_i$, for some form $\tilde\xi_i$ on $U_i$.
Let $\{B_i\}$ be a weakly good cover of $X$ that refines $\{U_i\}$.
For any pair $i,j$ with $B_i\subset U_j$ we, in abuse of notation, denote within the proof of this lemma $(k-1)$-form $\tilde \xi_j$ by $\xi_i$.
In this case we have 
$$ \| \xi_i \|_{L^p(B_i)} = \|\tilde\xi_j \|_{L^p(B_i)} \leq  \|\tilde\xi_j \|_{L^p(U_j)}\lesssim \|\omega \|_{L^p(X)}.$$
Note that $\xi_{i,j}:=\xi_i-\xi_j$ is a closed $0$-form on $B_i\cap B_j$ and therefore is a constant on $B_i\cap B_j$. Define 
$$ \xi:=\xi_i + c_i \text{\ \ on $B_i$}\ ,$$
where $c_i$ are constants such that $c_i-c_j=\xi_j-\xi_i$ on $B_i\cap B_j$.
Existence of such constants follows from the fact that $U$ is contractible. Indeed,
consider the 'nerve' complex $N_\bullet:=C_\bullet(\{B_i\})$ and let $f:N_1\to\R$ be $1$-cochain
defined by $f([ij]):=(\delta\xi)_{i,j}:=(\xi_j-\xi_i)|_{B_{i}\cap B_j}$. 
Clearly, $f$ is closed and since
$X$ is contractible $f$ is exact. Therefore $f=\delta g$ where $g$ is $0$-cochain of $K$.
Hence, $ (\xi_j-\xi_i)|_{B_{i}\cap B_j} = g(i)-g(j)$. Denote $C_i:=g(i)$.
These constants $C_i$ solve a system of linear equations 
$$(\delta c)_{i,j}=(\xi_j-\xi_i)|_{B_{i}\cap B_j}\ . $$ 
Therefore, 
$$C_i = \sum A_{i,j} (\xi_j-\xi_i)|_{B_{i}\cap B_j}\ ,$$
where $A_{i,j}\in \R$ are constants that depend only on the combinatorics of the cover $\{B_i\}$.
Below we estimate the $L^p$ norms of the constants $C_i$:

\begin{eqnarray*}
\| C_i\|_{L^p(B_i)} &\leq& \sum |A_{i,j}|\| (\xi_j-\xi_i)|_{B_{i}\cap B_j} \|_{L^p(B_i)}\\
&=& \sum |A_{i,j}|\| (\xi_j-\xi_i)|_{B_{i}\cap B_j} \|_{L^p(B_i\cap B_j)}
\left(\frac{Vol(B_i)}{Vol(B_i\cap B_j)}\right)^{1/p}\\
&\leq& 
\sum |A_{i,j}|\left\{\| \xi_j \|_{L^p(B_i)} + \|\xi_i\|_{L^p(B_j)}\right\}
\left(\frac{Vol(B_i)}{Vol(B_i\cap B_j)}\right)^{1/p}\\
&\lesssim& \| \omega \|_{L^p(X)}\ .
\end{eqnarray*}
It follows
$$ \|\xi \|_{L^p(X)}\leq \sum \|\xi_i+C_i \|_{L^p(B_i)}\leq \sum\|\xi\|_{L^p(B_i)}+\sum \|C_i\|_{L^p(B_i)}\lesssim\| \omega\|_{L^p(X)}\ ,$$
which completes the proof of in the case that $k=1$.

When $k>1$, our proof is similar to that of the case when $k=1$. Assume that 
$X = \cup_{i=1}^{N} \overline{U_i}$, where $\{U_i\}$ is a cover such that (\ref{loc_1}) holds for $U:=U_j$ with $\xi:=\tilde\xi_j$ for a form $\tilde\xi_j$. Similarly to the case of $k=1$ let $\{B_i\}$ be a weakly good cover that
refines $\{U_i\}$ and for any pair $i,j$ with $B_i\subset U_j$ we once again denote by $\xi_i$ the form $\tilde\xi_j$.

Once more, note that $\xi_{i,j}:=\xi_i-\xi_j$ is a closed $(k-1)$-form on $B_i\cap B_j$. Therefore, by 
the induction hypothesis, $\xi_{i,j}=d\xi^1_{i,j}$ and estimate (\ref{loc_1}) holds on $B_i\cap B_j$. 
From here, we can run the 'globalization' process as described in \cite{S}
to obtain solutions $\xi^{l+1}_I$ to the equations $(\delta \xi^l)_I = d\xi^{l+1}_I$ on $B_I$
(see Section $3.1$, Def. $3.8$ and Example $3.7$ illustrating all of the important features of the 'globalization' construction).
In the final step we have  a collection of $0$-forms $\xi^{k-1}_I$ with 
$(\delta\xi^{k-1})_I$ being constants. By an argument similar to the one in the case that $k=1$, with $f$ being a closed (and hence exact) cochain $f:N_{k}\to\R$, defined by $f([I]):=(\delta\xi^{k-1})_I$ there are constants $C_I$ such that
$$(\delta\xi^{k-1}+C)_J=0 $$ and, moreover,
$$\|C_I \|_{L^p(B_I)}\lesssim \|\omega \|_{L^p(X)}. $$

As is described in \cite{S} Section $3.2$, we may find a collection of $(k-1)$-forms $x^{k-1}_I$ defined
on $B_I$ such that $(\delta x^{k-1})_J=\xi^{k-1}_J-C_J$ and 
$\delta x^{k-t} = \xi^{k-t}-dx^{k-t+1}$ for $t>1$.
As in the proof of Proposition $3.12$ (replacing '$(p,q)$-Poincar\'e inequality for forms' by the local Poincar\'e $L^p$ inequality) it follows that the forms $x^{s}_I$ admit the following estimates:
$$\| x^s_I \|_{L^p(B_I)}\lesssim\| \omega\|_{L^p(X)}, $$
and
$$\| dx^s_J \|_{L^p(B_J)}\lesssim\| \omega\|_{L^p(X)}. $$
It is then straightforward to show that $\xi:=x^0$ is a global solution to problem (\ref{loc_1}).
\end{proof}

\begin{prop}
Assume $X$ is normal and $\U$ is a weakly good cover of $X$. Then there are
isomorphisms 
$$h_1 : H^j(K^\bullet(\U))\to H^j (K^{k,\bullet}(\U,\Omega^\bullet_{L^p}))$$
and 
$$h_2 : H^j(K^\bullet(\U))\to H^j (\Omega^\bullet_{L^p}(X)))$$
induced by the homomorphisms of the respective differential complexes. 
\end{prop}
\begin{proof}
When $X$ is a smooth manifold, constructions of $h_1$ and $h_2$ can be found in \cite{BT} (Theorem 8.1, Proposition 8.8 and Theorem 8.9).
To adapt these constructions in our setting one has to substitute 
the classical Poincar\'e lemma by Lemma \ref{loc_glob_1}, 
cf.  
the proof of our Proposition $3.14$ in \cite{S}
in which this construction is carried out in complete details.
\end{proof}

Denote by 
$$ Int:H^j (\Omega^\bullet_{L^p}(X))) \to H^j (C^\bullet(\U,\Omega^k_{L^p}))$$
the isomorphism $h_1\circ h_2^{-1}$.
\begin{rk}\label{eq_of_int} 
It is proved in \cite{S} that if $X$ is a compact manifold and $\omega$ is a closed smooth 
$k$-form on $M$ then $(Int\ \omega)c = \int_c \omega$ for every cycle $c$ in $X$.
Moreover, if $\U$ is a good cover and $c$ is a cycle given by $\sum_{I} a_I[I]$, then 
$$ (Int\ \omega)c = (-1)^{\lfloor{\frac{k}{2}\rfloor}}\sum a_I \delta\xi^{k-1}_I,$$
where  $\xi_I^{k-1}$ are the forms on $U_I$ constructed for every form $\omega$ 
by the inductive relation:
$$ d\xi^{s+1}_I = (\delta \xi^s)_I\ \text{ on } U_I, $$
where $\xi^0_{i_0}$ is a solution to $d\xi^0_{i_0}=\omega|^{}_{U_{i_0}}$ satisfying 
$\|\xi^0_{i_0} \|_{L^p(U_{i_0})}\lesssim \| \omega\|_{L^p(U_{i_0})}$ given by Lemma \ref{loc_glob_1}.

%
\end{rk}

As a consequence of Lemma \ref{loc_glob_1} and Remark \ref{eq_of_int}
one may extend the definition of an integral over the cycles in $X$ to all closed $L^p$ bounded
forms as follows 
$$ \int_c \omega := (Int\ \omega)c = (-1)^{\lfloor{\frac{k}{2}\rfloor}}\sum a_I \delta\xi^{k-1}_I,$$
where $c=\sum_I a_I [I]$ is a cycle in $X$.

\subsection{Global $L^p$ inequality.}
A global analog of problem (\ref{loc_ineq}) can be formulated
as follows. 
Say that $X$ satisfies the global $L^p$ inequality for forms if there exists 
a constant $C>0$ such that for every closed form $\omega\in\Lambda^{k}_{L^p}(X)$ 
with zero integrals over every cycle in $X$, there is a form $\xi\in\Lambda^{k-1}_{L^p}(X)$ such that
\begin{equation}\label{glob_ineq}
\left\{
\begin{array}{lc}
 \omega=d\xi\ \ \text{on $X_{reg}$}, \\ \|\xi\|_{L^p(X)}\leq C \| \omega \|_{L^p(X)}\ .
\end{array}
\right. 
\end{equation}

\begin{prop}\label{prop_glob_lp_ineq}
For sufficiently large $p>>1$ if $\omega$ is a closed $k$-form in $\Lambda^k_{L^p}(X)$ and $\int_{c} \omega = 0$ for every $c\in H_k(X)$ then
$\omega$ is exact and (\ref{glob_ineq}) holds for $\omega$.
\end{prop}
\begin{proof}

The proof of this proposition follows the same argument as the proof of 
Theorem $3.1$ in \cite{S}. We construct the form $\xi$ 
satisfying (\ref{glob_ineq}) following faithfully the structure of the construction in \cite{S}
 Sections $3.1$ and $3.2$ for a finite weakly good cover $\U$ of $X$,
but  replacing the '$(p,q)$ Poincar\'e inequality' for forms by Lemma \ref{loc_glob_1} .
\end{proof}

%
%
%
%
%
%
%
%
%
%
%
%
%

\section{Introduction to Lipschitz Retraction Theorem}\label{sec_Lip_retract}
Deformation retractions play an important role in De Rham theory.
For instance, a standard proof of classical  Poincar\'e lemma  on a star shaped
domain $U\subset\R^n$ uses a smooth deformation retraction 
$r:U\times I \to U$ to construct a primitive of a closed form $\omega$
in the following way. Let us assume for simplicity that $U$ is star shaped
(from $0\in\R^n$). Let $r_t(x)=r(x,t):=tx$. Assume that $\omega$ is a closed form, 
then we have the following (unique) decomposition of the pull back of $\omega$ by $r$:
$$ r^*\omega = \omega_0 + dt\wedge\omega_1, $$ 
where the differential forms $\omega_0$ and $\omega_1$ do not contain any terms
involving $dt$.
Set 
$$ \gamma(x):=\int_0^1 \omega_1(x,t) dt .$$
Now $d\gamma = \omega$. Indeed, since $d\omega=0$ and $d$ commutes with $r^*$ we have
$$ 0=dr^*\omega=d_x\omega_0 + dt\wedge(\frac{\pa\omega_0}{\pa t}-d_x\omega_1), $$
where $d_x$ represents the exterior derivative with respect to $x$.
Therefore, $\frac{\pa\omega_0}{\pa t}=d_x\omega_1$ and hence
$$ d\gamma(x) = \int_0^1 d_x \omega_1 dt = \int_0^1 \frac{\pa\omega_0}{\pa t} dt 
=\left. \omega_0(x,t)\right|_{t=0}^{t=1}=r_1^*\omega(x) - r_0^*\omega(x)=\omega(x) .$$

However, in this article we deal with semialgebraic sets $X$, which need
not have star shaped neighborhoods of every point.
Therefore, to extend the Poincar\'e lemma to our setting we will have to
construct Lipschitz semialgebraic deformation retractions with controlled growth of their derivatives.
The main techniques of our construction 
are based on Lipschitz semialgebraic geometry theory developed in \cite{V1}.



In what follows, we represent points $q\in\R^{n+1}$ by pairs $(x,y)\in\R^{n}\times \R$.

\begin{df}\label{def_cell}
A {\bf cell} in $\R^n$ is defined by induction on $n$. For $n=1$, a cell is a point
or an open interval. For $n>1$ a cell is either a graph of a semialgebraic function or
a band delimited by two semialgebraic functions over a cell in $\R^{n-1}$.

A cell is called {\bf Lipschitz cell} if all the graphs and bands involved in
its construction are defined by means of Lipschitz semialgebraic functions.

A {\bf cell subdivision} of $\R^n$ is a subdivision of $\R^n$ into a disjoint collection of cells.
A cell subdivision of $\R^n$ is said to be {\bf compatible} with a set $A$ if $A$ 
can be represented as a union of the cells of this subdivision.
\end{df}

For every collection of semialgebraic sets in $\R^n$ there exists a cell subdivision 
compatible with them.

\begin{thm}
Let $A_1,\dots, A_m\subset\R^n$ be semialgebraic sets. There exists a cell subdivision 
of $\R^n$ compatible with $A_i$ for $i=1,\dots,m$.
\end{thm}

In general, providing a cell subdivision is often not sufficient for a study of a semialgebraic set since cell subdivision 
does not include information on how the cells come in contact with the neighboring cells.
Let $\A$ be a collection of cells in $\R^n$. We say that $\A$ satisfies the
{\bf frontier condition} if the boundary of each cell in $\A$ is a union of cells in $\A$.
Next we introduce a concept of stratification. 
\begin{df}\label{df_strat}
A {\bf stratification} of a set $X$ is a collection $\Sigma$ of smooth manifolds called {\bf strata} such that their union is the set $X$ and the boundary of each stratum is union of the strata of lower dimension.

If $S$ and $S'$ are two strata in $\Sigma$ such that $S'\subset\partial S$ then we write $S'\leq S$. 

Denote by $\Sigma^k$ the collection of all strata in $\Sigma$ of dimension $k$, by $\Sigma^{(k)}$ the collection of all strata up to (and including) dimension $k$ and by $|\Sigma|$ the union of all strata in $\Sigma$.
A {\bf refinement} of $\Sigma$ is a stratification $\Sigma'$
such that each stratum of $\Sigma$ is a union of strata of $\Sigma'$. We then write $\Sigma'\prec\Sigma$. If $f:X\to Y$ is a map and $\Sigma$ is a stratification of $X$ then we
write $f(\Sigma)$ to denote the collection of sets $\{f(S):S\in\Sigma\}$.

We say that $\Sigma$ is a {\bf Whitney A} stratification if for every two strata $S'\leq S$
and a sequence of points $p_n\in S$ converging to $p\in S'$ we have
$\lim_{n\to\infty} T_{p_n}S \supset T_p S'$ whenever the limit on the left hand side exists.
\end{df}
Every semialgebraic set admits a Whitney A stratification, moreover

\begin{thm} 
Suppose that $X\subset \R^n$ is a semialgebraic set. There exists a Whitney A stratification 
of $\R^n$ compatible with $X$.
\end{thm}
For a proof see e.g., \cite{BCR}.

\begin{rk}\label{coords_on_cell}
A Lipschitz cell $C$ of dimension $k$ in $\R^{n+1}$ is bi-Lipschitz equivalent to a $k$-dimensional Lipschitz cell 
$D\subset\R^k$. A bi-Lipschitz homeomorphism $\phi: C \to D$
can be constructed by induction as follows. The cell $C$ is either a graph of a Lipschitz semialgebraic function or a band bounded by
two Lipschitz semialgebraic functions over a cell $C'\subset\R^{n}$. 
Assume $\phi': C' \to D'$ is a bi-Lipschitz homeomorphism, 
where $D'$ is a $(k-1)$-dimensional Lipschitz cell in $\R^{k-1}$. Now,

if $C$ is a graph of $\theta:\C'\to \R$ then set 
$$\phi(x,\theta(x)):=(\phi'(x),\theta(x)),$$

if $C$ is a band bounded by $\theta_i:\C'\to\R$, $i=1,2$, $\theta_1<\theta_2$
define
$$ \phi (x,y):=(\phi'(x),y)\ .$$
Since $\theta_1,\theta_2$ and $\theta$ are Lipschitz, the map $\phi$ is bi-Lipschitz.
Note that the map $\phi^{-1}$ defines coordinates $u=(u_1,\dots,u_k)$ on $C$.
Also, observe that a function $f:C\to\R$ is Lipschitz if and only if 
$f\circ\phi^{-1}$ is Lipschitz.
\end{rk}

\begin{df}\label{std_lift}
Suppose that $C$ is a cell of $\R^{n+1}$ given by a graph of 
a function $\theta$ or by a band bounded by graphs of functions $\theta_1$ and $\theta_2$ over a cell $C'\subset\R^n$.
Moreover, assume that we have a deformation retraction $r':C'\times I \to C'$.
The {\bf standard lift} of $r'$ is a deformation retraction $r:C\times I \to C$ 
defined by
$$ r_t(q):=r(q,t):=(r'(x,t), (1-\tau(q))\theta_1(r'(x,t)) + \tau(q)\theta_2(r'(x,t)))\ ,\ \ q=(x,y)\ ,$$
where $\tau(q):=\frac{y-\theta_1(x)}{\theta_2(x)-\theta_1(x)}$ in the case that
$C=\{q:\theta_1(x)<y<\theta_2(x),\ x\in C' \}$ and by
$$ r_t(q):=r(q,t):=(r'(x,t), \theta(r'(x,t)))\ ,$$
in the case that $C = \{q:y=\theta(x),\ x\in C'\}$. Note that $\tau(r(q,t))=\tau(q)$.
\end{df}

In the reminder of this section we give 
an intuitive derivation of our main Lipschitz deformation retraction Theorem \ref{lip_retraction_main}.
In what follows we describe a rough idea of our construction of a Lipschitz semialgebraic deformation retraction $r$ on a neighborhood of a point in a semialgebraic set.
We remark that the Lipschitz semialgebraic deformation retraction of Theorem \ref{lip_retraction_main}
has additional estimates on its derivatives, but for the sake
of simplicity we will only deal with the Lipschitz property of $r$ for now (i.e. in this section).

Assume that $C$ is a cell in $\R^{n+1}$ bounded by two Lipschitz semialgebraic functions
$\theta_1<\theta_2$ defined over a cell $C'\subset\R^n$ and that 
$r':C'\times I\to C'$ is a Lipschitz semialgebraic deformation retraction. It is not
always true that the standard lift of $r'$ to $C$ is Lipschitz as
the following example shows.
\begin{ex}
Let $\xi(x):\R^2\to\R$, $\xi(x)=|x_1^2-x_2|$, $r'_t(x):=tx$. 
Assume that $C$ is a cell in $\R^3$ defined by $\{(x,y)\in\R^2\times\R: 0\leq y \leq \xi(x) \}$.
Let $r_t$ be the standard lift of $r'$ from $\R^2$ to $\R^3$.
Clearly $\xi$ is Lipschitz. Let us show that $r_t$ is not Lipschitz.
Note that 
$$ r_t(x,y) = (tx,y\frac{\xi(tx)}{\xi(x)}). $$
Observe that $r_t$ is continuous and differentiable almost everywhere, so $r_t$ is Lipschitz if 
and only if all partial derivatives of its components are bounded.
In particular, if $r_t$ is Lipschitz then $\frac{\xi(tx)}{\xi(x)}$, being the derivative of the last component of $r_t$ with respect to $y$, has to be bounded. We will show that 
$\frac{\xi(tx)}{\xi(x)}$ is not bounded. Indeed, set 
$$ x_1 = t,\ \ \ \ x_2=t^2+t^5 $$
and observe that
$$ \frac{\xi(tx)}{\xi(x)}=\frac{|t^4-t^3-t^6|}{|t^5|} \to \infty \text { as } t\to 0.$$

It is possible to redefine the deformation retraction $r'$ on $\R^2$
in such a way that its standard lift would be Lipschitz. Indeed,
let $r'_t(x):=(tx_1,t^2x_2)$ then 
$$ \frac{\xi(r'_t(x))}{\xi(x)}=t^2,$$
and hence the standard lift $r_t$ is Lipschitz.
\end{ex}

This example leads us to formulate a condition for a standard lift 
of a Lipschitz semialgebraic deformation retraction to be Lipschitz.

\begin{prop}\label{lip_lift}
Assume that $C\subset\R^{n+1}$ is a cell which 
is a graph of a Lipschitz semialgebraic function $\theta_1$ or a band bounded by 
Lipschitz semialgebraic functions $\theta_2$ and $\theta_3$ over a cell $C'$ of $\R^n$. 
Let $r':C'\times I\to C'$ be a Lipschitz semialgebraic deformation retraction
and $r$ be its standard lift. 
The standard lift $r$ is Lipschitz in the case that $C$ is a graph .
When $C$ is a band, the standard lift $r$ is Lipschitz if and 
only if 
\begin{equation}\label{theta_esti_intro}
|\theta_2(r'_t(x))-\theta_3(r'_t(x))|\lesssim |\theta_2(x)-\theta_3(x)|,
\end{equation}
where $r'_t(x):=r'(x,t)$.
\end{prop}
\begin{proof}
Since semialgebraic functions are generically smooth we only have to check that 
partial derivatives of $r$ 
are bounded.
By Remark \ref{coords_on_cell} there exist 
bi-Lipschitz maps $\phi':C'\to D'$ 
and $\phi:C\to D$ such that the diagram

$$\begin{array}[c]{ccc}
C&\stackrel{\phi}{\rightarrow}&D\\
\  \ \ \ \downarrow\scriptstyle{\pi_{n+1}}&&\downarrow\scriptstyle{\pi_k}\\
C'&\stackrel{\phi'}{\rightarrow}&D'
\end{array}$$
is commutative, where $\pi_j:\R^{j}\to\R^{j-1}$ is the standard projection to the first $j-1$ coordinates.
Therefore, we may assume (by replacing $C$ with $D$ and $C'$ with $D'$) that $C'$ is a cell of dimension $n$ in $\R^n$ and $C$ is either a graph or a band over $C'$.

The map $r_t$ can be written as $ r_t(x,y) = (r'_t(x),r_{n,t}(x,y))$.
Set $D_j:=\frac{\pa}{\pa x_j}$ and $D_t:=\frac{\pa}{\pa t}$.
By our assumption $r'$ is Lipschitz. Therefore, we only have to check that
$|D_j r_{n,t}|$ and $\left|\frac{\pa r_{n,t}}{\pa t}\right|$ are bounded.

In the case that $C$ is a graph of $\theta_1$ we have
$$ |D_j r_{n,t}(x,y) | = |D_j (\theta_1(r'(x,t)))|= |\sum_i D_i\theta_1(r'(x,t))D_j r'_{i,t}(x)|, $$
which is bounded since $\theta_1$ and $r'_t$ are Lipschitz. For the same reason
$$|D_t r_{n,t}(x,y)|=
|\sum_j D_j\theta_1(r'(x,t))D_t r'_{j,t}(x)| $$ 
is bounded.

In the case that $C$ is a band bounded by $\theta_2<\theta_3$, set $\theta(x):=\theta_3(x)-\theta_2(x)$ and let $D_y:=\frac{\pa }{\pa y}$.
Let $\tau(q)=\frac{y-\theta_2(x)}{\theta(x)}$ be as in Definition \ref{std_lift}.
\begin{eqnarray*}
| D_j r_{n,t}(x,y)| &=& |D_j\left( \theta_2(r'(x,t)) + \tau(q)\theta(r'(x,t)) \right)|\\
&\leq& |D_j \theta_2(r'(x,t))| + |D_j\left(\tau(q)\theta(r'(x,t)\right)|\\
&\leq& C_1 + |(D_j\tau(q))\theta(r'(x,t))+\tau(q)D_j\theta(r'(x,t))|.
\end{eqnarray*}
Note that 
$$ D_j\tau(q) = \frac{-D_j\theta_2(x)\theta(x)
- (y-\theta_2(x))D_j\theta(x)}{\theta^2(x)}$$
Thus, by (\ref{theta_esti_intro}) and the fact that $0<y-\theta_2(x)<\theta(x)$ we have
$$
|D_j\tau(q)\theta(r'(x,t))| \leq |D_j\theta_2(x)| + |D_j\theta(x)| \leq C.
$$

The estimate of $| D_y r_{n,t}(x,y)|$ is obtained as follows.
\begin{eqnarray*}
| D_y r_{n,t}(x,y)| &=& |D_y\left( \theta_2(r'(x,t)) + \tau(q)\theta(r'(x,t)) \right)|\\
&=& |D_y\left(\tau(q)\theta(r'(x,t)\right)|\\
&=& |(D_y\tau(q))\theta(r'(x,t))+\tau(q)D_y\theta(r'(x,t))|\\
&=& |\frac{1}{\theta(x)}\theta(r'(x,t))|\leq C.\\
\end{eqnarray*}

We omit the proof of the boundedness of the partial derivative in $t$ of the standard lift $r$
because it is nearly identical to the proof of the boundedness of $D_j r$ above.

For the inverse implication in the case that $C$ is a band, 
assume that $r_t(x,y)$ is Lipschitz. 
It follows that $|D_y r_{n,t}|=|\frac{1}{\theta(x)}\theta(r'(x,t))|$ is bounded,
as required. 
\end{proof}

We will make use of the basic construction of a Lipschitz semialgebraic deformation retraction 
on a Lipschitz cell as a standard lift of a Lipschitz semialgebraic deformation from lower dimensional cell.
Our goal is to obtain a local Lipschitz semialgebraic deformation retraction to any point of $X$ as a step in
an inductive process. 
Criterion (\ref{theta_esti_intro}) derived in Proposition \ref{lip_lift} results in the following Lipschitz deformation retraction theorem.

\begin{thm}\label{lip_retraction_intro}
Let $X:=\cup_{j=1}^{m} X_j \subset\R^n$ be a closed semialgebraic set, $0\in \overline X_j\cap X$,$j=1,\dots,m$, and $\xi_1,\dots,\xi_s:\R^n\to \R$ some continuous semialgebraic functions.
Then there exists a neighborhood $U$ of $0$ in $\R^n$ and 
a Lipschitz deformation retraction $r:U\times I\to U$ that preserves $X_j$ , $j=1,\dots,m$
and satisfies
\begin{equation}\label{xi_esti_intro}
\xi_j(r_t(x))\lesssim \xi_j(x).
\end{equation}
\end{thm}
The latter theorem is too weak for our applications and is included here only 
as an introduction to the topic of the Lipschitz deformation retraction.
Therefore we only sketch its proof.
A 'stronger' version of this theorem is Theorem \ref{retract} which is 
proven in Section \ref{sec_Lip_ret} in complete details. 
\\
{\it Sketch of the proof of Theorem \ref{lip_retraction_intro}}. The proof is by induction on $n$. The case of $n=1$ is easy so we skip it 
and go directly to proving the inductive step.
First we use a preparation theorem for functions $\xi_j$ in combination
with a bi-Lipschitz transformation to bring the semialgebraic sets $X_i$ into a `good position'
and `prepare' the functions $\xi_j$. More precisely, after a bi-Lipschitz
transformation the topological boundaries of the sets $X_j$ will belong to 
the union of graphs of globally defined Lipschitz semialgebraic functions $\eta_1<\dots<\eta_b$
defined over $\R^n$ and the functions $\xi_j$ will be of the following form 
$$ \xi_j (q)\left. \right|_{C_i} \sim |y-\eta_{l_{i,j}}|^{w_{i,j}}a_{i,j}(x)\ , $$
where $\{C_i\}$ is a cell subdivision of $\R^{n+1}$ 
consisting of graphs and bands of functions $\eta_j$ over the cells in $\R^{n}$.
Next we apply the inductive hypothesis to the cells in $\R^n$ and the following collection of functions
$$\eta_j,\  |\eta_i-\eta_{j}|,\  \min\{|\eta_i-\eta_{j}|^{w_{i,j}}a_{i,j}(x),1\}\ \text{ for all } i,j\ . $$
 As an output of the inductive step we obtain a deformation retraction $r'$ on $\R^n$.
Set $r$ to be the standard lift of $r'$ (see Definition \ref{std_lift}).
Applying the criterion of Proposition \ref{lip_lift} 
the deformation retraction $r$ is Lipschitz. To complete the inductive step
we prove that condition (\ref{xi_esti_intro}) holds, namely:

\begin{lem}\label{check_xi_ineq}
Suppose that $C\subset\R^{n+1}$ is a cell bounded by graphs of Lipschitz semialgebraic functions $\theta_1<\theta_2$ over a cell $C'\subset\R^n$. Let $r'_t$ be a Lipschitz semialgebraic deformation 
retraction on $C'$ and $r$ be its standard lift.
Assume that $\zeta(q)=|y-\xi(x)|^wa(x)$, $w\in\Q$, is a bounded function such that 
$\xi\leq \theta_1$. Set 
$$ \theta(x):=|\theta_2(x)-\theta_1(x)| \text{ and } \eta(x):=|\xi(x)-\theta_1(x)|\ .$$
If 
\begin{eqnarray}\label{xi_ineq_11}
\min(a(z_1)\eta(z_1)^{w},1) &\lesssim&
\min(a(x)\eta(x)^{w},1),\\ \nonumber
\min(a(z_1)\theta(z_1)^{w},1)&\lesssim&
\min(a(x)\theta(x)^{w},1).           
\end{eqnarray}
then 
$$\zeta(r_t(q))\lesssim \zeta(q).$$
\end{lem}
\begin{proof}
Note that 
$$|y-\xi(x)|=|y-\theta_1(x)|+|\xi(x)-\theta_1(x)|\ .$$
Observe that
\begin{equation}\label{xi_mq_1}
\zeta(q)\sim a(x)\left\{ \begin{array}{ll}
                \min(|y-\theta_1(x)|^{w},
                \eta(x)^{w} )  &  w <0\\
                \max (|y-\theta_1(x)|^{w},
                \eta(x)^{w} )  &  w >0\ .
  
               \end{array} \right.
\end{equation}
Let $z:=z(t)=(z_1(t),z_2(t))$ be the components of the deformation retraction $r=r(q,t)$ where $(z_1(t),z_2(t))\in \R^n\times\R$, $q=(x,y)$. To simplify the notation we will write $(z_1,z_2)$ instead of $(z_1(t),z_2(t))$.
From Definition \ref{std_lift} and (\ref{xi_mq_1}) it follows that: 
\begin{equation}\label{xi_m11}
\zeta(z)= a(z_1)\left\{ \begin{array}{ll}
                 \min\{|\tau(z)\theta(z_1)|^{w}\},
                 \eta(z_1)^{w}\}  &  w <0\\                
                 \max\{|\tau(z)\theta(z_1),\eta(z_1)^w\} & w >0\ .\\
           \end{array} \right.
\end{equation}

Note that $\tau(q)=\tau(z)$.
\newline
If ${w}<0$ then, since $\zeta$ is bounded, it follows 
\begin{equation}\label{xi_m_le01}
\zeta(z) \sim \min\{ \min(a_k(z_1)|\tau(z)\theta(z_1))|^{w},1),
\min(a(z_1)\eta(z_1)^{w},1) \}\ .
\end{equation}

Note that if condition (\ref{xi_ineq_11})  holds for $f_1$ and $f_2$ then it also
holds for $\min\{f_1,f_2\}$. Also if $f$ is a non-negative and bounded function then $f\sim \min(f,1)$.
Therefore, it suffices to prove that 
$$\min(a(z_1)|\tau(z)\theta(z_1))|^{w},1)\lesssim 
\min(a(x)|\tau(q)\theta(x)|^{w},1),$$
and 
$$\min(a(z_1)\eta(z_1)^{w},1) \lesssim
\min(a(x)\eta(x)^{w},1)\ , $$

The latter inequality is a straightforward consequence of our assumption.
For the former inequality, we note that 
$$
\min(a(z_1)\theta(z_1)^{w},1)\lesssim
\min(a(x)\theta(x)^{w},1),           
$$
Therefore,
\begin{eqnarray}\label{xi_m2_2}
\min(a(z_1)|\tau(z)\theta(z_1))|^{w},1) &=&
\min\{\tau(z)^{w} a(z_1)\theta(z_1)^{w},\tau(z)^{w},1\}\nonumber\\&=&
\min\{\tau(z)^{w}\min( a(z_1)\theta(z_1)^{w},1),1\}\nonumber\\&\lesssim&    
\min\{\tau(q)^{w} a(x)\theta(x)^{w},1\}\nonumber\ .         
\end{eqnarray}

Assume now that $w>0$. It follows from the fact that $\zeta$ is bounded, formula (\ref{xi_m11}) and the our assumption
\begin{equation}\label{xi_mge0}
a(z_1)|\tau(z)\theta(z_1)|^{w} \lesssim
a(x)|\tau(q)\theta(x)|^{w}\ .
\end{equation}

Therefore,
\begin{eqnarray*}
\zeta(z) &\sim&
\max\{ (a(z_1)|\tau(z)\theta(z_1)|^{w},
                 a(z_1)\eta(z_1)^{w} \} \\
                 &\lesssim&
\max(a(x)|\tau\theta(x)|^{w},
                 a(x)\eta(x)^{w} )\\ 
                 &\sim&     \zeta(q)\ .
\end{eqnarray*}

\end{proof}


\begin{rk}
The main result of this section is a strengthening of Theorem \ref{lip_retraction_intro}, in which we construct a deformation retraction with various estimates
on its derivatives in terms of the deformation parameter $t$.  This topic is technically the most important part of our work.
\end{rk}

\begin{thm}(Lipschitz deformation retraction theorem)\label{lip_retraction_main}
Let $\Sigma_0$ be a stratification of $\R^n$, $X=\cup X_j$, $X_j\in\Sigma_0$, $0\in\overline X_j\cap X$, $j=1,\dots,m$.
There exist a stratified neighborhood $(U,\Sigma_U)$ of $0$ in $\R^n$ with $\Sigma_U$ a cell subdivision such that $\Sigma_U\prec\Sigma\cap U$ and a Lipschitz semialgebraic deformation retraction $r:U\times I\to U$
such that 
\begin{enumerate}
\item $r_0(x)=0$, $r_1(x)=x$ ,
\item $r|_{S\times(0,1]}$ is smooth ,
\item $|\det Dr_t|\gtrsim t^\mu$, for some $\mu\geq 0$,
\item $\|Dr_t \|\lesssim t^\lambda$ for some $\lambda>0$,
\end{enumerate}
where  $\det Dr_t$ is taken with respect to the coordinates of the respective cell of $\Sigma_U$ 
(see Remark \ref{coords_on_cell}) and $\|Dr_t \|$ denotes the operator max-norm of the tangent map.
\end{thm}

We prove this theorem in Section \ref{sec_Lip_ret}.

\begin{rk}\label{on_xi_ineq}
In contrast, condition (\ref{xi_esti_intro}) and the functions $\xi_j$ of Theorem \ref{lip_retraction_intro} are absent in the statement of 
Theorem \ref{lip_retraction_main}:
condition (\ref{xi_esti_intro}) was included only to carry the inductive step of the proof of
Theorem \ref{lip_retraction_intro} and not 
for our applications. The proof of Theorem \ref{lip_retraction_main} is by induction similar to that of the proof of Theorem \ref{lip_retraction_intro}.
It involves functions analogous to the functions $\xi_j$ of Theorem \ref{lip_retraction_intro} and inequalities strengthening condition (\ref{xi_esti_intro}).
\end{rk}




\section{Regular families of hypersurfaces and bi-Lipschitz homeomorphisms}\label{sec_reg_fam}
This section contains preliminaries made use of in Section \ref{sec_Lip_ret} in order to construct 
a Lipschitz semialgebraic deformation retraction on a semialgebraic set with control
on the growth of the derivatives. 

In our sketch of proof of Theorem \ref{lip_retraction_intro} we did not include  
an explanation on how to construct a bi-Lipschitz transformation of the ambient $\R^{n+1}$ that maps a given set with empty interior to a subset of a union of a finite number of graphs of Lipschitz semialgebraic functions. Construction of such a bi-Lipschitz map was essentially obtained by G.Valette
in \cite{V1}, but we include it for the sake of completeness. 

In our main Theorem \ref{bi_Lip_cone_2} of this section we 
start with a stratification $\Sigma$ of $\R^{n+1}$ and construct a bi-Lipschitz 
transformation $h$ of $\R^{n+1}$ that maps a given cone, to a perhaps larger cone, such that 
the restriction of $h$ to a certain refinement $\A\prec\Sigma$  is a diffeomorphism.
Moreover, the images of the strata in $\A$ are graphs of Lipschitz semialgebraic functions 
or bands over cells in $\R^n$.

In what follows we will use the following notations. Let $e_i$ , $i=1,...,n$ to be the standard basis and $S^{n-1}$ the unit sphere of $\R^n$.
For $\lambda \in S^{n-1}\subset \R^n$, 
we denote by $N_\lambda$ the orthogonal to $\lambda$ subspace of $\R^n$ (shortly $N_\lambda:=\lambda^{\perp}$) and 
by $\pi_\lambda$ the projection onto
$N_\lambda$ along $\lambda$. 
Given $q \in \R^n$, we denote by $q\mapsto q_\lambda$ the standard  scalar product (in $\R^n$)  with  $\lambda$.
We say that a set $H\subset\R^{n+1}$ is a graph {\bf relative to $\lambda$} if there exists a function
$\xi : N_\lambda \to \R $ such that
$$ H=\{ q\in\R^{n+1} : q_\lambda=\xi(\pi_\lambda(q)) \} .$$


\begin{df} A {\bf Lipschitz cell decomposition of $\R^n$} is a cylindrical cell decomposition $\C$ of $\R^n$ which is also a stratification and is such that for $n>1$ each cell $C\in\C$ is either a graph of a Lipschitz semialgebraic function or a band bounded by two Lipschitz semialgebraic functions over some cell $C'$ in $\R^{n-1}$.
The vector $e_n$ is said to be {\bf regular} for $\C$ if for each cell $C\in\C$
the restriction to $C$ of $\pi_n:=\pi_{e_n}$ is a one-to-one map and, also, there exists a Lipschitz function $\xi:\pi_n(C)\to \R$
such that $C$ is the graph of $\xi$ over $\pi_n(C)$.
\end{df}

We will need the following results from [V1] for our construction of the bi-Lipschitz transformation of this section.

\begin{df}\label{family regulier d hypersurface}
A {\bf regular family of hypersurfaces} of $\R^{n+1}$ is a family
$H=(H_k;\lambda_k)_{1 \leq k \leq b}$ , $b\in\N$, of hypersurfaces of
$\R^{n+1}$ together with elements $\lambda_k$ of $S^n$ such that the following
properties hold for each $k < b $:

\begin{itemize}
 \item[(i)]  The consecutive pairs $H_k$ and $H_{k+1}$ are  the
 graphs
 relative to $\lambda_k$ of two global Lipschitz functions $\xi_k$ and, respectively, $\xi'_k$
 such that $\xi_k \leq \xi'_k$ ;
\item[(ii)]  $E(H_{k+1};\lambda_k)=E(H_{k+1};\lambda_{k+1}), $
where 
$E(H_k;\lambda_k)=\{ q\in\R^{n+1} : q_\lambda \leq \xi( \pi_\lambda (q)) \} \ .$
\end{itemize}
Let $A$ be a semialgebraic subset of $\R^{n+1}$ of empty interior. We say that
the family $H$ is {\bf compatible} with $A$, if $A \subseteq
\bigcup_{k=1} ^{b} H_k$. An {\bf extension} of $H$ is a regular
family compatible with the set $\bigcup_{k=1} ^b H_k$.
\end{df}


We  will  also make use of the following notation
\begin{notation}
Let $\lambda\in S^{n-1}$ and $M\in[0,1)$. 
We denote
$$ C_{n}(\lambda,M):=\{q\in\R^n : \frac{q\cdot\lambda}{|q|}\geq M\}\subset \R^n\ $$
(which are cones centered at $0$ with 
the axis being $\lambda$).
\end{notation}
Given two functions $f,g: A\to\R$ we say that $f$ is {\bf equivalent} to $g$,  $f\sim g$,
if there exist $c_1>0$ and $c_2>0$ such that $c_1f\leq g\leq c_2f$.
If $f \leq c_1 g$, we write $f\lesssim g$. 
We say that $f$ is {\bf comparable} with $g$ if
the difference $f-g$ has a constant sign.


\begin{thm} \label{prop_3_10} {( Proposition 3.10 [V1])}
For each semialgebraic set $A\subset\R^{n+1}$ with empty interior and $\eps>0$  there exists
a regular family $(H_k;\lambda_k)_{1 \leq k \leq b}$ of hypersurfaces of $\R^{n+1}$ compatible with $A$ and such that $\lambda_k\cdot e_{n+1}>1-\eps$, $1\leq k \leq b$ .
\end{thm}

Given a regular system of hypersurfaces $H:=\hk$ in $\R^{n+1}$ 
we associate with it a bi-Lipschitz map $h_H:\R^{n+1}\to\R^{n+1}$ (Proposition 3.13 [V1]). 
For completeness we give the construction of $h_H$ below.

\begin{prop}\label{prop313}(Proposition 3.13 [V1] )
Let $H:=\hk$ be a regular system of hyperplanes in $\R^{n+1}$.
There exists a bi-Lipschitz mapping $h_H:\R^{n+1}\to\R^{n+1}$ 
that maps each hypersurface $H_k$  
to a hypersurface $F_k$ which is a graph of a Lipschitz semialgebraic function $\eta_k$ for $e_n$.
\end{prop}
\begin{proof}
 We define $h_H$
over $E(H_k;\lambda_k)$  by induction on $k$ in such a way that
\begin{equation}\label{star1}
h_H(E(H_k;\lambda_k))=E(F_k ;e_n)\ ,
\end{equation} 
where
$F_k$ is the graph of a Lipschitz function $\eta_k$ relative to $e_n$\ . 
Note that then $h_H(H_k)=F_k$.

For $k=1$ choose an orthonormal basis in $N_{\lambda_1}$ and set
$h_H(q)=(x_{\lambda_1};q_{\lambda_1})$, where $x_{\lambda_1}$
 are the coordinates of $\pi_{\lambda_1}(q)$ in this basis.
Then, let $k \geq 1$ and assume that $h_H$ has been already
constructed on $E(H_k;\lambda_k)$. 
Property $(i)$ of Definition \ref{family regulier d hypersurface} 
says that $H_k$ and
$H_{k+1}$ are the graphs relative to the same $\lambda_k$ of two Lipschitz
functions $\zeta_k$ and $\zeta'_k$. For $q \in E(H_{k+1};\lambda_{k})
\setminus E(H_k;\lambda_k)$ set 
$$h_H(q):=h_H(\pi_{\lambda_k}(q);\zeta_k \circ \pi_{\lambda_k}(q))
+(q_{\lambda_k}-\zeta_k \circ \pi_{\lambda_k}(q))e_n\ .$$

Due to property $(ii)$ of Definition \ref{family regulier
d hypersurface} we have $ E(H_{k+1};\lambda_{k+1})=
E(H_{k+1};\lambda_{k})$, so that  $h_H$ turns out to be defined over $E(H_{k+1};\lambda_{k+1})$. Since $\zeta_k$ is Lipschitz $h_H$ a
bi-Lipschitz homeomorphism. Note also that (\ref{star1}) holds with $F_{k+1}$ a graph of the following Lipschitz function
$$\eta_{k+1}(q)=\eta_k \circ \pi_{e_n}(q)+(\zeta'_k-\zeta_k)
\circ \pi_{\lambda_k}\circ h^{-1}(q;\eta_k \circ \pi_{e_n}(q))\ .$$
We now constructed $h_H$ over $E(H_b ; \lambda_b)$. To extend $h_H$ to the
whole of $\R^n$ we follow the case of $k=1$ (use $\lambda_b$ instead of $\lambda_1$). Now it is straightforward to verify that $h_H$ is a bi-Lipschitz homeomorphism.
\end{proof}


The following theorem is the main theorem of this section (cf. \cite{V2} Corollary 2.2.4).
\begin{thm} \label{bi_Lip_cone_2}
Let $\Sigma$ be a stratification of $\R^{n+1}$ compatible with a cone $C:=C_{n+1}(e_1,M)$.
There exists a refinement $\A\prec \Sigma$ and a bi-Lipschitz map $h:\R^{n+1}\to\R^{n+1}$
such that 
\begin{enumerate}[(1)]
\item $h|_A$ is diffeomorphism for all $A\in\A$,
\item  there exists $0<M'<1$ such that $h(C)\subset C_{n+1}(e_1,M')$,
\item  every stratum $h(A)$, $A\in\A$, is either a graph of a Lipschitz semialgebraic function 
or a band bounded by graphs of Lipschitz semialgebraic functions over a stratum 
$h(A')$ in $\R^n$, $A'\in\A$.
\end{enumerate}
\end{thm}

For our proof the following proposition is crucial.  

\begin{prop}\label{cone_to_cone}
Let $A\subset\R^{n+1}$ be a semialgebraic set and let $H=\hk$
be a regular system of hyperplanes compatible with $A$ such that $\lambda_k\cdot e_{n+1}>1-\eps$.
If $0\in A\subset C_{n+1}(e_1,M)$ then there exists $M'$ such that $h_H(A)\subset C_{n+1}(e_1,M')$,
 where $h_H$ is provided by Proposition \ref{prop313}.
\end{prop}

Our proof of the latter proposition will make use of the following 3 lemmas.
\begin{df}
Assume that $S\subset\R^{n+1}$ is a graph of a function $\xi$ relative to $\lambda\in S^{n}$.
Let map $\pi_{S}:\R^{n+1}\to S$ to be defined by
$\pi_{S}(q):=\pi_{\lambda}(q) + \xi(\pi_{\lambda}(q))\lambda$, 
\end{df}

\begin{lem}\label{lip_cone}
Let $A'\subset C_{n}(e_1,M)$ and $\xi:A'\to \R$, $\xi(0)=0$ be a 
Lipschitz semialgebraic function with Lipschitz constant $L$.
Then, $\Gamma_\xi(A')\subset C_{n+1}(e_1,M/(1+L))$.
\end{lem}
\begin{proof}
We have to show that for $x\in A'$
$$ \frac{(x,\xi(x))\cdot e_1}{\sqrt{\sum x_i^2 + \xi(x)^2}} \geq M/(1+L)\ . $$
Since $\xi(0)=0$ we have 
$$|\xi(x)|=|\xi(x)-\xi(0)| \leq L|x|\ . $$ 
Therefore
$$ \sqrt{\sum x_i^2 + \xi(x)^2} \leq \sqrt{\sum x_i^2} + |\xi(x)| \leq (1+L)|x|\ .  $$
Since $x\in C_n(e_1,M)$ it follows that 
$$\frac{(x,\xi(x))\cdot e_1}{\sqrt{\sum x_i^2 + \xi(x)^2}} \geq \frac{x\cdot e_1}{(1+L)|x|} \geq  \frac{M}{1+L}\ .$$
\end{proof}

\begin{df}
Assume that $\lambda\in S^{n}$, $\lambda\cdot e_{n+1}>0$. Let $e_{j,\lambda}$ be the 
unique vector in $N_\lambda$ such that $\pi_{n+1}e_{j,\lambda} = e_j $.
\end{df}


\begin{lem} \label{cone_incl}
Let $v:=e_{1,\lambda}\in N_\lambda\subset\R^{n+1}$ and $\eps>0$. 
If  $\lambda\cdot e_{n+1}\geq 1 -\eps$ then 
$$C_{n+1}(e_1,M)\subset C_{n+1}(v/|v|,M')$$
with $M':=\frac{1}{|v|}\left( M - \frac{\sqrt{2\eps}}{1-\eps}\right).$
\end{lem}
\begin{proof}
Since $v$ projects to $e_1$ there exists a constant $A\in\R$ such that
$v=e_1+Ae_{n+1}$.
Since $N_\lambda = \lambda^{\perp}$ it follows that
$0=\lambda\cdot v = e_1\cdot\lambda + Ae_{n+1}\cdot\lambda\  .$ 
Hence
$ A = \frac{-e_1\cdot\lambda}{e_{n+1}\cdot\lambda}\ .$
Of course
$\lambda=\sqrt{1-\delta^2}e_{n+1}+\delta w$ 
with $w\perp e_{n+1}$, $|w|=1$ and $\delta\in\R_+$. Consequently, $ 1-\eps\leq \lambda\cdot e_{n+1} = \sqrt{1-\delta^2},$
\begin{eqnarray*}
1+\eps^2-2\eps &\leq& 1-\delta^2\\
\delta^2 &\leq& 2\eps-\eps^2. 
\end{eqnarray*}
Therefore, $\delta\leq \sqrt{2\eps}$ and the estimate for $|A|$ follows:
$$|A| = \left|\frac{-e_1\cdot\lambda}{e_{n+1}\cdot\lambda}\right| 
\leq \left|\frac{|\delta w\cdot e_1|}{1-\eps} \right| 
\leq \left|\frac{\sqrt{2\eps}}{1-\eps} \right|\ .$$

Finally, let $q\in C_{n+1}(e_1,M)$. Then $q\in C_{n+1}(v,M')$ due to
\begin{eqnarray*}
\frac{q\cdot v}{|q||v|} &=& \frac{q\cdot e_1}{|q||v|} + A\frac{q\cdot e_{n+1}}{|q||v|}\\
&\geq&\frac{1}{|v|}\left( M - \frac{\sqrt{2\eps}}{1-\eps}\right)\ ,
\end{eqnarray*}
as required.
\end{proof}

\begin{lem}\label{proj_to_hyper}
Let $A\subset C_{n+1}(e_{1,\lambda},M)$, $\lambda\cdot e_{n+1}>1-\eps$ and let $H_1\subset \R^{n+1}$ be a graph of a Lipschitz 
semialgebraic function $\xi$, $\xi(0)=0$ for $\lambda_1$, $\lambda_1\cdot e_{n+1}>1-\eps$. 
Then, $\pi_{H_1}(A)\subset C_{n+1}(e_{1,\lambda_1},M')$.
\end{lem}
\begin{proof}
It follows from Lemma \ref{cone_incl} that 
$C_{n+1}(e_{1,\lambda},M)\subset C_{n+1}(e_{1,\lambda_1},M'')$ for some $M''$. Hence $\pi_{\lambda_1}(A)\subset C_{n+1}(e_{1,\lambda_1},M'')$ and since $\xi$ is Lipschitz with $\xi(0)=0$ it follows that $\pi_{H_1}(A)\subset C_{n+1}(e_{1,\lambda_1},M')$, where $M':=M''/(1+L))$ and $L$ is the Lipschitz constant of $\xi$.
\end{proof}

{\it Proof of Proposition \ref{cone_to_cone}}.
Assume first that $A\subset H_k$ for some $k$
and $H_k$ is a graph of $\xi_k$ relative to $\lambda_k$.
Observe that 
$$\pi_{n+1}\circ (h_H |_{H_k})(q) = \pi_{\lambda_1}\circ\pi_{H_1}\circ\dots\circ\pi_{H_{k-1}}(q).$$

It follows by iterating Lemma \ref{proj_to_hyper}, that $\pi_{n+1}\circ (h_H |_{H_k})(A)\subset C_{n+1}(e_1,M'')$ for some $M''$.
Since $h_H(H_K)=F_k$ is a graph of a Lipschitz semialgebraic function $\eta_k$ with $\eta_k(0)=0$ we conclude using Lemma \ref{lip_cone} that $h_H(A)\subset C_{n+1}(e_1,M''/(1+L_k))$, where
$L_k$ is the Lipschitz constant of $\eta_k$.
In the general case of $A\subset\cup H_k$ we apply the argument as above to $A\cap H_k$ and then conclude 
that $h_H(A)\subset \cup C_{n+1}(e_1,M''/(1+L_k))\subset C_{n+1}(e_1,M')$, as required.
$\hfill \square$\\

Nect we prove Theorem \ref{bi_Lip_cone_2}.\\

{\it Proof of Theorem \ref{bi_Lip_cone_2}.}
Let $H=\hk$ be a regular system of hyperplanes compatible
with the topological boundaries of $\Sigma$. 
Let $h_H:\R^{n+1}\to\R^{n+1}$ be the bi-Lipschitz map given by Proposition \ref{prop313}
and let $\eta_k:\R^n\to\R$ be Lipschitz semialgebraic functions such that $F_k:=h_H(H_k)$ is a graph of $\eta_k$.

By Proposition \ref{cone_to_cone} we have $h_H(C)\subset C_1:=C_{n+1}(e_1,M')$ and
hence conclusion (2) is proven.

Next, we construct a stratification $\A$ of $\R^{n+1}$ such that 
$h|_A$ is a diffeomorphism for every $A\in \A$ and moreover,
$h(A)$ is either included in a graph of one of the $\eta_i$'s or is a band
bounded by the graphs of two consecutive $\eta_i$'s.

By induction on $k$ we define a family of stratifications $\F_k:=\{\A_{1,k},\dots,\A_{k,k} \}$ 
such that for each $i$,  $\A_{i,k}$  is a stratification of $H_i$ that refines $\A_{i,k-1}$. 

For $k=1$ stratification $\A_{1,1}$ is a stratification of $H_1$.
Assuming that $\F_k$ is constructed we construct $\F_{k+1}$ as follows.

Define $\A_{k+1-j,k+1}$ by induction on $j$.
For $j=0$ set $\A_{k+1,k+1}$ to be a stratification of $H_{k+1}$.
Assuming that $\A_{k+1-j,k+1}$ is constructed we construct $\A_{k-j,k+1}$.
Since the hypersurface $H_{k+1-j}$ is a graph of a Lipschitz semialgebraic function relative to $\lambda_{k-j}$,
we set $\A_{k-j,k+1}$ to be a refinement of $\A_{k-j,k}$ compatible with all $\pi_{\lambda_{k-j}}^{-1}(A)\cap H_{k-j}$ for $A\in\A_{k-j+1,k+1}$.

Consequently, the final family $\F_b$ consists of a stratification of the hypersurfaces $\hk$
which induces a stratification $\A$ of $\R^{n+1}$ with the strata of $\A$ being:
\begin{itemize} 
\item The strata of each $\A_{j,b}$, $j=1,\dots,b$
\item The bands bounded by the graphs of $\zeta_k$ and $\zeta'_k$ relative to $\lambda_k$ intersected with 
$\pi_{\lambda_k}^{-1}$(A), where $A\in \A_{k,b}$ and $k=1,\dots,b$.
\item $\{q:q_{\lambda_{b}}>\zeta_b(q)\}$ and $\{q:q_{\lambda_{1}}<\zeta_1(q)\}$
\end{itemize}
Our construction of $h$ clearly guarantees that $h|_A$ is smooth for all $A\in\A$ and 
that $\Sigma_1:=\{h(A)\}_{A\in\A}$ forms a stratification of $\R^{n+1}$.
%
%
%
%
%
%
%
%
%
$\square$

%
%
%
%
%
%
%
%
%
%

\section{Proof of Lipschitz deformation retraction theorem}\label{sec_Lip_ret}
In this section we prove the main technical result of our work, Theorem \ref{lip_retraction_main}.
The proof follows the same structure as the sketch of the proof of Theorem \ref{lip_retraction_intro}.

Recall that inequality 
$\xi_j(r'_t(x))\lesssim \xi_j(x) $
with $\xi_j$ being a 
difference of two Lipschitz functions that bound a cell, is a criterion for the 
standard lift $r_t$ of $r'_t$ to be Lipschitz. In order to show that
$Dr_t$ admits estimates in terms of $t$, we will have to enlarge
inequality (\ref{xi_esti_intro}) from Theorem \ref{lip_retraction_intro} 
to a group of several inequalities:
\begin{enumerate}
  \item\begin{enumerate}[(a)]
          \item $\xi_j(r_t(q))\lesssim \xi_j(q)$
          \item if $\xi_j(0)=0$ then
                 $\xi_j(r_t(q))\lesssim t^{\lambda_j}\xi_j(q)$ for $q\in X$, $\lambda_j>0$.
      \end{enumerate}
\item $\xi_j(r_t(q)) \gtrsim t^{\mu_j}\xi_j(q) $, $\mu_j\geq 0$
\end{enumerate}
Then, we derive the estimates of $|\det Dr_t|$ and $\|Dr_t\|$ in terms of $t$ by making use of the
latter inequalities in a way similar to our proof of $r$ being Lipschitz. 


In the sketch of proof of Theorem \ref{lip_retraction_intro} we remarked that 
the functions $\xi_j$ of the theorem may be assumed (upon a bi-Lipschitz transformation) to be of the form $\xi_j(q)=|y-\eta_j(x)|^{w_j}a_j(x)$, where $\eta_j$ are Lipschitz and $a_j$ are continuous semialgebraic functions. We will make use of the following theorem and lemma (both from \cite{V1}) to justify this assumption.

\begin{thm}\label{lem function eq aux distances}{(Proposition 4.3 [V1])}
Given a non negative semialgebraic function $f$ on $\R^n$, there exist a finite number of semialgebraic
subsets $W_1, \dots,W_s$, and a partition of $\R^n$ such that $f$
is equivalent to a product of powers of distances to the $W_j$'s
on each element of the partition.
\end{thm}

\begin{lem}\label{prep_lemma}
Let $\eta_1,\dots,\eta_b:\R^n\to \R$ be Lipschitz semialgebraic functions
and $V\subset \R^{n+1}$ be a cell bounded by two consecutive $\eta_i$'s over $V':=\pi_{n+1}(V)$.
Suppose that  $W_1,\dots,W_m\subset \cup \Gamma_{\eta_j}$ and $\xi:V\to \R_+$ is a bounded semialgebraic function  equivalent to 
$$\prod{d(\cdot,W_j)^{\alpha_j}},\ \alpha_j\in \Q$$
Then, there exist Lipschitz semialgebraic functions $\theta_1 \leq \dots\leq \theta_{b'}:\R^n\to \R$, 
continuous semialgebraic functions $a_i:V\to \R$ 
and a subdivision $V=\cup V_i$ such that 
each $V_i$ is a cell bounded by two consecutive $\theta_i$'s over $V'$ and, moreover,
\begin{equation}\label{xi_m_prepared1}
\xi(q)|_{V_i}\sim |y-\eta_{\nu_i}(x)|^{w_i}a_i(x)\ .
\end{equation} 
\end{lem}

We prove this lemma following
\begin{rk}\label{r11}
${\ }$
\begin{enumerate}
 \item
 If $A$ is a union of graphs of  semialgebraic functions $\theta_1 , \dots , \theta_k$ over $\R^n$  
then there is an ordered family of semialgebraic functions  $\xi_1 \leq \dots \leq\xi_k$ such that $A$ is a union of graphs of these functions.
\item
Given a family of  Lipschitz semialgebraic functions $f_1, \dots, f_k$
defined over $\R^n$ there is a cell decomposition $\C'$
of $\R^n$ and some Lipschitz semialgebraic functions $\xi_1\leq \dots \leq  \xi_m$
on $\R^n$ such that over each cell
$C=\{ q=(x\ ;\ q_{n+1})\in\R^{n+1} :\ x\in C' ,\ \xi_i(x)\leq q_{n+1} \leq \xi_{i+1}(x)\}$, 
where $C'\in\C'$, the semialgebraic functions $|q_{n+1}-f_i(x)|$  are comparable with each other
and are comparable with the functions $f_i \circ \pi_n$.
Indeed, it suffices to consider the  graphs of
functions  $f_i$, $f_i+f_j$  and  $\frac{f_i+f_j}{2}$
and then the family $\xi_1, \dots, \xi_m$  is provided by the first remark.
\end{enumerate}
\end{rk}

\begin{proof}[Proof of Lemma \ref{prep_lemma}]
Note that 
\begin{equation}\label{split_sum}
d(q, W_{j}\cap\Gamma_{\eta_\nu})\sim |y-\eta_\nu(x)|+
 d(x, \pi_n(W_{j}\cap\Gamma_{\eta_\nu}))\ ,\ \  q\in\R^{n+1}\ , 
\end{equation}
where $j\in J$ and $1\leq\nu\leq b$. 
%
%
According to Remark \ref{r11}, there is a collection of Lipschitz semialgebraic functions  
$\{\theta_1,\dots,\theta_{b'}\} \supset \{\eta_1,\dots,\eta_b\}$ on $\R^n$ such that 
there exists a cell decomposition $\C_0$ of $\R^{n+1}$ with the following properties:
\begin{enumerate}
\item The cells of $\C_0$ consist of graphs and bands of $\theta_i$'s over the cells in $\R^n$.
\item The semialgebraic functions $d(x, \pi_n(W_j\cap\Gamma_{\eta_\nu}))$, 
$\eta_\nu$, $|\eta_\nu-\eta_{\nu'}|$ , and $|y-\eta_\nu|$, where $1\leq\nu,\nu'\leq b$ and $j\in J$, are pairwise comparable over each cell in $\C_0$.
\end{enumerate}
Let $\C$ be a stratification of $\R^{n+1}$ obtained from $\C_0$ by refining the cells in $\R^n$ in such a way that taking graphs and bands of the restrictions of $\theta_j$'s to those cells
results in a Whitney A stratification of $\R^{n+1}$. 



Let $C$ be an open cell in $\R^{n+1}$ bounded by the graphs of $\theta_{j_0}$ and
$\theta_{j_0+1}$ over a cell $C'$ in $\R^n$. Due to the fact that the cell decomposition $\C$ is compatible with the graphs of the $\eta_i$'s, 
we have either $\eta_i|_{C'}\geq\theta_{j_0+1}$ or $\eta_i|_{C'}\leq\theta_{j_0}$ for any $i\in\{1,\dots,b\}$.
Note that for any $j\in J$ and $q\in C$  we have
$$ d(q,W_j)=\min_\nu d(q,W_j\cap\Gamma_{\eta_\nu})\ . $$
Therefore,
\begin{eqnarray*}
\xi(q) &\sim& \prod_{j\in J} (\min_\nu d(q,W_j\cap\Gamma_{\eta_\nu}))^{w_{j}}\\ 
&\sim&
\prod_{j\in J} \left(|y-\eta_\nu(x)|+ d(x, \pi_n(W_j\cap\Gamma_{\eta_\nu}))\right)^{w_{j}}\ .  
\end{eqnarray*}
Each expression of the form
$$ \left(|y-\eta_\nu(x)|+
d(x, \pi_n(W_{j}\cap\Gamma_{\eta_\nu}))\right)^{w_{j}} $$
is equivalent to
\begin{equation}\label{a_le_0}
 \min\left( |y-\eta_\nu(x)|^{w_{j}},
d(x, \pi_n(W_{j}\cap\Gamma_{\eta_\nu}))^{w_{j}} \right)
\quad\text{   if ${w_{j}}<0$ } 
\end{equation}
and is equivalent to  
\begin{equation}\label{a_ge_0}
 \max\left( |y-\eta_\nu(x)|^{w_{j}},
d(x, \pi_n(W_{j}\cap\Gamma_{\eta_\nu}))^{w_{j}} \right)
\quad\text{   if ${w_{j}}>0$\ . } 
\end{equation}

Since over the cell $C$, functions from collection 
$\{|y-\eta_\nu(x)|$, $|\eta_\nu-\eta_{\nu'}|, d(x, \pi_n(W_{j}\cap\Gamma_{\eta_\nu}))\}$
with $1\leq\nu,\nu'\leq b$, $i\in I$ and $j\in J$, are pairwise comparable it follows that the expressions in (\ref{a_le_0}) and (\ref{a_ge_0}) 
are equal to either $|y-\eta_\nu(x)|^{w_{j}}$ or 
$d(x, \pi_n(W_{j}\cap\Gamma_{\eta_\nu}))^{w_{j}}$.
Also, one of the following 3 properties holds
\begin{itemize}
\item $|y-\eta_\nu(x)|\sim |y-\eta_{\nu'}(x)|$ 
\item $|y-\eta_\nu(x)|\sim |\eta_\nu(x)-\eta_{\nu'}(x)|$
\item $|y-\eta_{\nu'}(x)|\sim |\eta_\nu(x)-\eta_{\nu'}(x)|$ . 
\end{itemize}
Consequently, there are constants $\nu_k,{w}$, ${w_{\nu,\nu'}}$, and ${w'_{j}}$
such that over the cell $C$ (\ref{xi_m_prepared1}) holds with
$$a(x)=\prod_{\nu,{\nu'}}|\eta_\nu-\eta_{\nu'}|^{w_{\nu,\nu'}}
\prod_{j\in J} d(x, \pi_n(W_{j}\cap\Gamma_{\eta_\nu}))^{w'_{j}}\ .$$

\end{proof}

Next we prove Theorem \ref{lip_retraction_main} in a formulation convenient for a proof by induction on the dimension of the ambient $\R^n$.
\begin{thm}\label{retract}
Let $\Sigma_0$ be a stratification of $\R^n$ compatible with $X:=\cup_{j\in J} X_j$, $X_j\in\Sigma_0$ (with $J$ a finite index set) such that $0\in \overline {X_j}\cap X$ for all $j\in J$ and $X\subset C_n (e_1,M)$.
Let $\xi_1,\dots,\xi_l:\R^n\to \R_+$ be bounded semialgebraic functions.
Then, 
there are an open stratified neighborhood $(U,\Sigma_U)$ of $0$ in $\R^n$, $\Sigma_U\prec \Sigma_0$ and a Lipschitz semialgebraic deformation 
retraction $r:U\times [0,1]\to U$ such that 
\begin{enumerate}
\item $r|_{S\times (0,1]}$ is smooth for all $S\in \Sigma_U$ ;
\item $r:S\times (0,1] \to S$, $S\in \Sigma_U$, $r_0(q)=0$,  $r_1(q)=q$ ;
\item\label{xi_ineq} \begin{enumerate}[(a)] 
          \item\label{xi_ineq_1} $\xi_j(r_t(q))\lesssim \xi_j(q)$ ;
          \item\label{xi_ineq_2} if $\xi_j$ is continuous near $0$ and $\xi_j(0)=0$ then
                 for some $\lambda_j>0$ and all $q\in X$ inequality $\xi_j(r_t(q))\lesssim t^{\lambda_j}\xi_j(q)$ holds;
      \end{enumerate}
\item\label{xi_ineq_3} $\xi_j(r_t(q)) \gtrsim t^{\mu_j}\xi_j(q) $ for some $\mu_j\geq 0$ ;
\item\label{r_ineq_1} $|\det D r_t|\gtrsim t^\mu$ for some $\mu\geq 0$ ;
\item\label{r_ineq_2} $\| Dr_t|_{X_{reg}} \| \lesssim t^{\lambda}$ for some $\lambda>0$ .
\end{enumerate}

\end{thm}
%
%
%
%
\begin{proof}
Assume $n=1$. Let $r(x,t):=tx$. Each function 
$\xi_j$ is a bounded semialgebraic function near $0$ and therefore is continuous at $0$. 
If $\xi_j(0)=a_j > 0$ then for $x$ small enough we have $\xi_j\sim a_j$ and therefore
estimates (\ref{xi_ineq_1}), (\ref{xi_ineq_3}), (\ref{r_ineq_1}) and (\ref{r_ineq_2})  hold.
If $\xi(0)=0$ then 
expanding each  $\xi_j$ into a Puiseux series 
$$ \xi_j(x) =  b_j x^{w_j} + R(x),\ \   w_j\in\Q_{+},\ \  R(x)\in o(x^{w_j}), $$ 
it follows
$$ \xi_j(x) \sim  b_j x^{w_j}. $$ 
Consequently all estimates from (\ref{xi_ineq}) to (\ref{r_ineq_2}) follow.
Next we prove
 $\mathbf{(H_{n+1})}$ assuming $\mathbf{(H_{n})}$. 
Throughout the proof, we represent points of $\R^{n+1}$ by $q=(x,y)\in\R^{n}\times\R$.
According to Theorem \ref{lem function eq aux distances} there exists a finite partition $\{V_i\}_{i\in I}$
of $\R^{n+1}$ and a finite family of closed semialgebraic subsets $\{W_{j}\}_{j\in J}$ with empty interiors (otherwise, we would replace such $W_j$'s by their topological boundaries),  
such that for every $V_i$ with $0$ in its closure holds:
\begin{equation}\label{prod_dist}
\xi_k(q)\sim \prod_{j\in J}{d(q,W_j)^{w_{ijk}}}\ ,\ \ q\in V_i\ ,  
\end{equation}
where $1\leq k\leq l$ and $w_{ijk}\in\Q$.
We may assume that $0\in W_{j}$ for all $j\in J$ since the $W_{j}$ that do not contain $0$ are superfluous for the validity of (\ref{prod_dist}).
Consider a stratification $\tilde\Sigma$ that simultaneously refines stratification $\Sigma_0$
and both collections $\{V_i\}_{i\in I}$ and $\{W_j\}_{j\in J}$.
Applying Theorem \ref{bi_Lip_cone_2} with input $\tilde{\Sigma}$ results in a refining stratification $\A$ and a bi-Lipschitz map $h:\R^{n+1}\to \R^{n+1}$ 
such that all of the images of the topological boundaries of semialgebraic sets in $\Sigma_0\cup\{V_i\}_{i\in I}\cup \{W_j\}_{j\in J}$ are included in 
the graphs of Lipschitz semialgebraic functions $\eta_1,\dots,\eta_b:\R^n\to \R$ and 
$h(X)\subset C_{n+1}(e_1,M')$ for some $M'\in[0,1)$.
Moreover, $h|_A$ is a diffeomorphism for all $A\in\A$ and each stratum from $\A$ is mapped
to a stratum of the form of a graph over lower dimensional stratum of one of the $\eta_j$'s or
a band bounded by two consecutive $\eta_j$'s.

Since the final conclusion of Theorem \ref{retract} is stable upon application of a semialgebraic bi-Lipschitz homeomorphism we may identity  map $h$ with the
identify the map and rename $M'$ by $M$.


Lemma \ref{prep_lemma} applied to the sets $W_j$ and functions $\xi_k$ 
results in a collection of Lipschitz semialgebraic functions $\theta_1\leq\dots\leq \theta_{b'}$
(including functions $\eta_j$) defined over $\R^n$  and a cell decomposition $\Sigma\prec \A$ of $\R^{n+1}$ 
such that the cells in $\R^{n+1}$ are given by graphs and bands of $\theta_i$'s over the cells
in $\R^n$ and over each cell $C\in \Sigma$
\begin{equation}\label{xi_m_prepared}
\xi_k|^{}_{C}(q)\sim |y-\eta_{k,C}(x)|^{w_{k,C}}a_{k,C}(x)\ .
\end{equation}

We assume without loss of generality that $X_j$'s are cells of $\Sigma$ since otherwise,
we may achieve this by refining the stratification $\Sigma$.
Note that $X'\subset C_n(e_1,M')$.
Apply the inductive hypothesis to the cells $\Sigma'$ (cells of $\Sigma$ in $\R^n$) with $X':=\cup \pi_n(X_j)$ and to a collection of semialgebraic functions $\mathcal{G}$ which we list following. The functions in $\mathcal{G}$ are the following:
\begin{itemize}
\item $d(x,0)$
\item $|\theta_j(x)|$ for $1\leq j\leq b'$,
\item $|\eta_j(x)-\eta_{j+1}(x)|$ for ${1\leq j\leq b-1}$,
\item $|\theta_j(x)-\theta_{j+1}(x)|$ for ${1\leq j\leq b'-1}$,  
\item $\min\left(a_k(x)|\theta_j(x)-\theta_{j+1}(x)|^{w^{}_{k}},1\right)$ 
for ${1\leq j\leq b'-1}$, ${1\leq k\leq l}$ , 
\item $\min\left(a_k(x)|\theta_j(x)-\eta^{}_{\nu^{}_k}(x)|^{w^{}_{k}},1\right)$ 
for ${1\leq j\leq b'}$,  ${1\leq k\leq l}$,
\item $\min\left(a_k(x)|\theta_{j+1}(x)-\eta^{}_{\nu^{}_k}(x)|^{w^{}_{k}},1\right)$ 
for ${1\leq j\leq b'-1}$, ${1\leq k\leq l}$. 
\end{itemize}
As a consequence we obtain a neighborhood $U'$ of $0\in\R^n$ and Lipschitz semialgebraic deformation retraction 
$r':U'\times I\to U'$ with all the properties listed in the theorem.
Let $r$ be the standard lift of $r'$ as defined in Definition \ref{std_lift}.
Note that $r$ is continuous and smooth on $C\times(0,1]$ for every cell $C\in\Sigma$.
It is clear from the definition of $r$ that conclusions $(1)$ and $(2)$ hold for $r$.
We have to prove that $r$ is Lipschitz and that the estimates from (\ref{xi_ineq}) to (\ref{r_ineq_2}) hold.
Note that it is enough to prove all these estimates on a cell $C\in\Sigma$ of $X$,
so let $C$ be a cell in $X$.

{\it Proof that $r$ is Lipschitz. } 
If $C$ is a graph of a Lipschitz semialgebraic function then by Proposition \ref{lip_lift} 
the standard lift $r$ of $r'$ is Lipschitz.
If $C$ is a band bounded by $\theta_j\leq\theta_{j+1}$ then
let $\theta:=|\theta_j-\theta_{j+1}|$ and note that $\theta\in\mathcal{G}$. Therefore
we have $\theta(r'(x,t))\lesssim \theta(x)$. 
According to the criterion of \ref{lip_lift} the standard lift $r$ is also Lipschitz.

\textit{Proof of the estimates (\ref{xi_ineq_1}) and (\ref{xi_ineq_3}).}
%
%
When $C$ is a graph over a cell of $\Sigma'$ condition (\ref{xi_ineq}a) is straight forward consequence of the induction hypothesis. Otherwise, the cell $C$ is bounded by $\theta_j$  and $\theta_{j+1}$ and each $\xi_k$, $1\leq k\leq s$, is of the form (\ref{xi_m_prepared}). Moreover, due to the construction preceding \ref{xi_m_prepared} also either $\eta_{\nu_k}\leq\theta_j$ or $\eta_{\nu_k}\geq\theta_{j+1}$. We may assume,
without loss of generality that the former case holds. 
Consequently  
$$|y-\eta_{\nu_k}(x)|=|y-\theta_j(x)|+|\eta_{\nu_k}(x)-\theta_j(x)|\ .$$
Denote 
$$ \theta(x):=|\theta_{j+1}(x)-\theta_j(x)| \text{ and } \eta(x):=|\eta_{\nu_k}(x)-\theta_j(x)|\ .$$
Then 
\begin{equation}\label{xi_mq}
\xi_k(q)\sim a_k(x)\left\{ \begin{array}{ll}
                \min(|y-\theta_j(x)|^{w^{}_{k}},
                \eta(x)^{w_{k}} )  &  \text{ if \ } w^{}_{k} <0\\
                \max (|y-\theta_j(x)|^{w^{}_{k}},
                \eta(x)^{w_{k}} )  &  \text{ if \ } w^{}_{k} >0\ .
  
               \end{array} \right.
\end{equation}
Denote by $z:=z(t)=(z_1(t),z_2(t))$ the components of the deformation retraction $r=r(q,t)$ with $(z_1(t),z_2(t))\in \R^n\times\R$ and $q=(x,y)$. In abuse of notation we will write $(z_1,z_2)$ instead of $(z_1(t),z_2(t))$.
In terms of Definition \ref{std_lift} and due to (\ref{xi_mq}) it follows that: 
\begin{equation}\label{xi_m1}
\xi_k(z)= a_k(z_1)\left\{ \begin{array}{ll}
                 \min\{|\tau(z)\theta(z_1)|^{w_{k}},
                 \eta(z_1)^{w_{k}}\}  & \text{ if \ } w_{k} <0\\                
                 \max\{|\tau(z)\theta(z_1)|^{w_k},
                 \eta(z_1)^{w_{k}} \} & \text{ if \ } w_{k} >0\ .\\
           \end{array} \right.
\end{equation}

Recall (Def.\ref{std_lift}) that $\tau(r(q,t))=\tau(q)$ for all $t\in[0,1]$.
\newline
If ${w_k}<0$ then boundedness of $\xi_k$ implies 
\begin{equation}\label{xi_m_le0}
\xi_k(z) \sim \min\{ \min(a_k(z_1)|\tau(z)\theta(z_1))|^{w_k},1),
\min(a_k(z_1)\eta(z_1)^{w_k},1) \}\ .
\end{equation}

Note that if conditions (\ref{xi_ineq_1}) and (\ref{xi_ineq_3}) of the inductive hypothesis hold for $f_1$ and $f_2$ then they also
hold for $\min\{f_1,f_2\}$ and that if $f$ is a non-negative and bounded function then $f\sim \min(f,1)$.
Therefore, it suffices to prove that 
$$\min(a_k(z_1)|\tau(z)\theta(z_1))|^{w_k},1)\lesssim 
\min(a_k(x)|\tau(q)\theta(x)|^{w_k},1),$$
$$ \min(a_k(z_1)|\tau(z)\theta(z_1))|^{w_k},1)\gtrsim 
t^{\mu_k}\min(a_k(x)|\tau(q)\theta(x)|^{w_k},1)
$$
and that
$$\min(a_k(z_1)\eta(z_1)^{w_k},1) \lesssim
\min(a_k(x)\eta(x)^{w_k},1)\ , $$
$$\min(a_k(z_1)\eta(z_1)^{w_k},1) \gtrsim
t^{\mu_k}\min(a_k(x)\eta(x)^{w_k},1)\ , $$

The latter two are immediate consequences of the inductive hypothesis. 
For the proof of the former two inequalities we note that the inductive hypothesis implies
$$
\min(a_k(z_1)\theta(z_1)^{w_k},1)\lesssim
\min(a_k(x)\theta(x)^{w_k},1)\ ,           
$$
and also that 
$$
\min(a_k(z_1)\theta(z_1)^{w_k},1)\gtrsim
t^{\mu_k}\min(a_k(x)\theta(x)^{w_k},1)\ .           
$$
Therefore,
\begin{eqnarray}\label{xi_m2}
\min(a_k(z_1)|\tau(z)\theta(z_1))|^{w_k},1) &=&
\min\{\tau(z)^{w_k} a_k(z_1)\theta(z_1)^{w_k},\tau(z)^{w_k},1\}\nonumber\\&=&
\min\{\tau(z)^{w_k}\min( a_k(z_1)\theta(z_1)^{w_k},1),1\}\nonumber\\&\lesssim&    
\min\{\tau(q)^{w_k} a_k(x)\theta(x)^{w_k},1\}\nonumber\ .         
\end{eqnarray}
And similarly,
\begin{eqnarray}\label{xi_m2}
\min(a_k(z_1)|\tau(z)\theta(z_1))|^{w_k},1) &=&
\min\{\tau(z)^{w_k} a_k(z_1)\theta(z_1)^{w_k},\tau(z)^{w_k},1\}\nonumber\\&=&
\min\{\tau(z)^{w_k}\min( a_k(z_1)\theta(z_1)^{w_k},1),1\}\nonumber\\&\gtrsim&    
\min\{\tau(q)^{w_k}t^{\mu_k}\min( a_k(x)\theta(x)^{w_k},1),1\}\nonumber\\&\gtrsim&    
t^{\mu_k}\min\{\tau(q)^{w_k}\min( a_k(x)\theta(x)^{w_k},1),1\}\nonumber\\&=&    
t^{\mu_k}\min\{\tau(q)^{w_k} a_k(x)\theta(x)^{w_k},1\}\nonumber\ .         
\end{eqnarray}

Assume now that ${w_k}>0$. Boundedness of $\xi_k$, formula (\ref{xi_m1}) and the 
induction hypothesis imply that
\begin{equation}\label{xi_mge0}
a_k(z_1)|\tau(z)\theta(z_1)|^{w_k} \lesssim
a_k(x)|\tau(q)\theta(x)|^{w_k}\ .
\end{equation}
$$
a_k(z_1)|\tau(z)\theta(z_1)|^{w_k} \gtrsim
t^{\mu_k}a_k(x)|\tau(q)\theta(x)|^{w_k}\ .
$$
Therefore,
\begin{eqnarray*}
\xi_k(z) &\sim&
\max\{ (a_k(z_1)|\tau(z)\theta(z_1)|^{w_k},
                 a_k(z_1)\eta(z_1)^{w_k} \} \\
                 &\lesssim&
\max\{ a_k(x)|\tau\theta(x)|^{w_k},
                 a_k(x)\eta(x)^{w_k} \}\\ 
                 &\sim&     \xi_k(q)\ .
\end{eqnarray*}
Similarly
\begin{eqnarray*}
\xi_k(z) &\sim&
\max\{ (a_k(z_1)|\tau(z)\theta(z_1)|^{w_k},
                 a_k(z_1)\eta(z_1)^{w_k} \} \\
                 &\gtrsim&
t^{\mu_k}\max\{ (a_k(x)|\tau(q)\theta(x)|^{w_k},
                 a_k(x)\eta(x)^{w_k} \} \\
                &\sim&    
                t^{\mu_k} \xi_k(q)\ .
\end{eqnarray*}

\textit{Proof of the estimate (\ref{xi_ineq_2}).}
Assume that $\xi_j$ is continuous near $0$ and $\xi_j(0)=0$.
Note first that we may assume that $\xi_j(q):=d(q,0)$.
Indeed, according to the \L ojasiewicz inequality (Theorem 2.6.6 \cite{BCR}) there is a continuous semialgebraic function $\tilde\xi_j$
and $L>0$ such that 
$$ \xi_j(q)=\tilde\xi_j(q)d(q,0)^L. $$
We may assume 
that $\tilde\xi_j$ is one of the functions $\xi_i$ (by including it to begin with in the original collection)  and then (\ref{xi_ineq_1}) would imply  
$\tilde\xi_j(r_t(q))\lesssim\tilde\xi_j(q)$. Therefore $d(r_t(q),0)\lesssim t^{\lambda_d}d(q,0)$ for some $\lambda_d>0$ would imply
$$\xi_j(r_t(q))\lesssim t^{\lambda_d L}\xi_j(q). $$

Assume $C$ is a cell in $X$.\\
{\it Case 1:} $C$ is a graph of $\theta_j$ over $C'$. Then since $X\subset C_{n+1}(e_1,M)$
it follows that $\theta_j(0)=0$.
Note that $d(q,0)\sim d(x,0)$. Indeed, assume $L_j>0$ is the Lipschitz constant of $\theta_j$
$$ d(q,0) \sim d(x,0) + |\theta_j(x)-\theta_j(0)|\leq (1+L_j)d(x,0).$$
On the other hand,
$ d(q,0) \geq d(x,0)$. Therefore,  $d(q,0)\sim d(x,0)$, as we claimed.
Since the distance to zero function $x\mapsto d(x,0)$,  is in $\mathcal{G}$  the induction hypothesis implies that 
$d(r'_t(x),0)\lesssim t^{\lambda_{d'}} d(x,0) .$
Therefore,
$$ d(r_t(q),0) \sim d(r'_t(x),0)\lesssim t^{\lambda_{d'}} d(x,0) \sim t^{\lambda_{d'}} d(q,0) .$$
{\it Case 2:} Assume that $C$ is a band bounded by the graphs of $\theta_{j}<\theta_{j+1}$ over $C'$.
%
%
%
%
%
%
%
%
%
Note that 
$ d(q,0)\leq d(x,0)+ |y|$. But $|y|\leq C d(x,0)$ since $q\in C_{n+1}(e_1,M)$ and hence
$d(q,0)\lesssim d(x,0)$.
On the other hand  $d(q,0)\geq d(x,0)$ and therefore, $d(q,0)\sim d(x,0)$.
Our proof of the remainder of the estimate is as in case 1 .

\textit{Proof of the estimate (\ref{r_ineq_1}).}
Note that $\det Dr_t$ is well defined over each cell $C$. Indeed, 
every $C$ is constructed iteratively as a graph or a band over 
a cell in a lower dimension. Assume $x_C:=(x_{j_1},\dots,x_{j_k})$ are the coordinates 
on the band $C$. Then the remaining coordinates of $\R^{n+1}$ are Lipschitz semialgebraic functions of the coordinates $x_C$. 
Let $\pi : C \to D$ be the projection onto the $x_C$ coordinate subspace on $\R^{n+1}$
and $\phi:D\to C$ be its inverse. Then $\det Dr_t:=\det D\left(\pi\circ r_t\circ\phi\right)$.

Assume $C$ is a graph of $\theta_j$ over a cell $C'$.
Then $\det Dr_t = \det Dr'_t$ and therefore (\ref{r_ineq_1}) holds by
the inductive assumption.

Now assume $C$ is a band bounded by the graphs of $\theta_j$ and $\theta_{j+1}$ over
a cell $C'$. 
Then 
$$ r_{t,{n+1}} (q) = \theta_j(r'_t(x)) + \frac{(y-\theta_j(x))(\theta_{j+1}(r'_t(x))-\theta_{j}(r'_t(x)))}
{\theta_{j+1}(x)-\theta_{j}(x)}\ . $$
Note, that due to the iterative definition of $r$ as the standard lifts from the lower dimensions
the matrix $Dr_t$ is lower triangular. In particular,
$$ |\det Dr_t| = \left|\det Dr'_t\right| \left|\frac{\pa r_{t,n+1}}{\pa y}\right| .$$
Finally, the inductive hypothesis implies 
$$ \left|\frac{\pa r_{t,n+1}(q)}{\pa y}\right| = \left|\frac{\theta_{j+1}(r'_t(x))-\theta_{j}(r'_t(x))}{\theta_{j+1}(x)-\theta_{j}(x)}\right|\gtrsim t^{\mu_{\theta_j}} $$
for some $\mu_{\theta_j}\geq 0$ and therefore, 
$$ |\det Dr_t| \gtrsim t^{\mu'+\mu_{\theta_j}}\ ,$$
as required. 

\textit{Proof of the estimate (\ref{r_ineq_2}).}
It suffices to show that 
$$\left|\frac{\pa r_{t,n+1}}{\pa x_j}\right|\lesssim t^\lambda $$
and, also, that
$$\left|\frac{\pa r_{t,n+1}}{\pa y}\right|\lesssim t^\lambda $$
for some $\lambda>0$ . When $C$ is a graph of $\theta_j$ over $C'$ it follows due to the construction of the standard lift
$$ r_{t,n+1}(q) = \theta_j(r'_t(x))\ .$$
(and, in particular, does not depend on $y$). Therefore,
$$\left|\frac{\pa r_{t,n+1}}{\pa x_l}\right| = 
\left|\sum_{i} \frac{\pa \theta_j(r'_t(x))}{\pa x_i}\frac{\pa r'_{t,i(x)}}{\pa x_l}  \right|.$$
By the induction hypothesis, $\left|\frac{r'_{t,i(x)}}{\pa x_l}  \right|\lesssim t^{\lambda'} $
for some $\lambda'>0$. 
Since $\theta_j$ is a Lipschitz semialgebraic function it follows $\left|\frac{\pa \theta_j(r'_t(x))}{\pa x_i} \right|\lesssim 1$ 
which completes the proof of (\ref{r_ineq_2}) in the case of $C$ being a graph over $C'$.

Now assume $C$ is a band bounded by graphs of $\theta_j$ and $\theta_{j+1}$ over $C'$.
Define $\theta(x):=\theta_{j+1}(x)-\theta_j(x)$.
Then 
$$ r_{t,n+1}(q) = \theta_{j}(r'_t(x))+\left(y-\theta_j(x)\right)\frac{\theta(r'_t(x))}{\theta(x)}\ .$$

The latter and the inductive assumption imply that  
$$ \left|\frac{\pa r_{t,n+1}}{\pa y}\right| = \left|\frac{\theta(r'_t(x))}{\theta(x)}\right|\lesssim t^{\lambda_{\theta}} $$
and that
\begin{eqnarray*}
\left|\frac{\pa r_{t,n+1}}{\pa x_l}\right| &\leq&
\left|\sum_{i} \frac{\pa \theta_j(r'_t(x))}{\pa x_i}\frac{\pa r'_{t,i(x)}}{\pa x_l}  \right|
+ \left|\frac{\pa \theta_j(x)}{\pa x_l}\frac{\theta(r'_t(x))}{\theta(x)} \right| + \\
&+& 
\left(y-\theta_j(x)\right)\left|\frac{\theta(x) \sum_{i} \frac{\pa \theta(r'_t(x))}{\pa x_i}\frac{\pa r'_{t,i(x)}}{\pa x_l}-\theta(r'_t(x))\frac{\pa\theta(x)}{\pa x_l}}{\theta^2(x)} \right|.
\end{eqnarray*}
Of course since $C\subset X$ is in a cone it follows $\theta_j(0)=\theta_{j+1}(0)=0$
 and $0\leq y-\theta_j(x)\leq \theta(x)$ for $(x,y)\in C$.
Finally, the induction hypothesis implies
\begin{eqnarray*}
\left(y-\theta_j(x)\right)\left|\frac{\theta(x) \sum_{i} \frac{\pa \theta(r'_t(x))}{\pa x_i}\frac{\pa r'_{t,i(x)}}{\pa x_l}-\theta(r'_t(x))\frac{\pa\theta(x)}{\pa x_l}}{\theta^2(x)} \right| &\leq&
\left|{ \sum_{i} \frac{\pa \theta(r'_t(x))}{\pa x_i}\frac{\pa r'_{t,i(x)}}{\pa x_l}}\right|\\ &+& \left|\frac{\theta(r'_t(x))\frac{\pa\theta(x)}{\pa x_l}}{\theta(x)} \right|\\
&\lesssim& t^{\lambda'} + t^{\lambda_{\theta}}
\end{eqnarray*}
for some $\lambda_\theta>0$ and (\ref{r_ineq_2}) follows
in the case of $C$ being a band over $C'$ as well,
 which completes the proof of Theorem \ref{retract}.
\end{proof}


\begin{thebibliography}{VVV}

\bibitem [BCR]{BCR}  J. Bochnak,M.Coste and M.-F.Roy, Real Algebraic Geometry. Ergebnisse der Math. 36, Springer-Verlag(1998). 


\bibitem[BoMi]{BoMi}
L.P. Bos and P.D. Milman, Sobolev-Gagiliardo-Nirenberg and Markov type inequalities on subanalytic domains, Geometric And Functional Analysis, Vol. 5, No. 6 (1995).


\bibitem[BT]{BT}Bott, Tu, Differential forms in algebraic topology. Graduate Texts
in Mathematics, Vol. 82, Springer-Verlag, New York, 1982, xiv + 331 pp.

\bibitem [Ch]{Ch} J. Cheeger, On the Hodge Theory of Riemannian pseudomanifolds, Amer. Math. Soc. Proc. Sym. Pur. Math. XXXVI (1980) 91ñ146.

\bibitem [GKS]{GKS}
V.M. Golídshtein, V.I. Kuzíminov, I.A. Shvedov,
The integration of differential forms of classes W(p,q) (Russian). Sibirsk.
Mat. Zh., 1982, 23, 5, 63-79.(Engl. transl.: Siberian Math. J., 23, (1982), 5,
640-653).



\bibitem [GKS2]{GKS2} V. M. Gol'dshtein, V. I. Kuz'minov and I. A. Shvedov, a property of De Rham regularization operators, Sib. Mat. Vol. 25, No. 2, pp. 104-111, 1983.

\bibitem [GKS3]{GKS3} V. M. Gol'dshtein, V. I. Kuz'minov and I. A. Shvedov, $L_p$-cohomology of warped cylinders, Sib. Mat. Vol. 31, No. 6, pp. 55-63, 1990.


\bibitem[Gr]{Gr} M. Gromov, Asymptotic invariants of infinite groups, in "Geometric Group Theory", Vol. 2 (Sussex, 1991), London Mathematical Society Lecture Note Series, 182, Cambridge University Press, Cambridge, 1993, pp. 1ñ295  


\bibitem [H]{H} A. Hatcher, Algebraic topology, Cambridge University Press (2002).

\bibitem[HP]{HP} W. C. Hsiang, V. Pati, $L^2$-cohomology of normal algebraic surfaces. I. Invent. Math. 81 (1985), no. 3, 395ñ412. 

\bibitem[IwLu]{IwLu} {key-2} T. Iwaniec and A. Lutoborski, Integral Estimates for Null Lagrangians,
Arch. Rational Mech. Anal. 125 (1993) 25-79. Springer-Verlag 1993.

\bibitem [S1]{S} L. Shartser, Explicit proof of Poincar\'e Inequality for differential forms on manifolds, C. R. Math. Rep. Acad. Sci. Canada Vol. 33(1) 2011, pp 21-32. 
\bibitem [SV]{SV}  L.Shartser and G.Valette, De Rham theorem for $L^\infty$ forms and homology on singular spaces, C. R. Math. Acad. Sci. Soc. R. Canada Vol. 32(1) 2010, pp 23-32 
\bibitem[V1]{V1} G. Valette, Lipschitz triangulations, Illinois
Journal of Math., Vol. 49, No 3, Fall 2005, P 953-979
\bibitem[V2]{V2} G. Valette, $L^\infty$ cohomology is intersection cohomology,	arXiv:0912.0713v1 .
\bibitem [O]{O}   T.Ohsawa, On the $L^2$ cohomology of complex spaces Math.Z. 209 (1992), 
519-53O.

\bibitem [W]{W}  H.Whitney, Geometric integration Theory. Princeton, Princeton University Press, 1957. 15+387 pp.
\bibitem [Y]{Y}  B. Youssin,  $L^p$-cohomology of cones and horns,  J. Differential Geom. Volume 39, Number 3 (1994), 559-603. 
\end{thebibliography}
\end{document}